\documentclass[oneside,a4paper,11pt,notitlepage]{article}
\usepackage[left=0.8in, right=0.8in, top=1.5in, bottom=1.5in]{geometry}

\usepackage[T1]{fontenc} % codifica dei font
\usepackage[utf8]{inputenc} % lettere accentate da tastiera
\usepackage[english]{babel}
\usepackage{amssymb}
\usepackage{lipsum} % genera testo fittizio
\usepackage{lmodern}
\usepackage{amsmath}
\usepackage{amsthm}
\usepackage{bm}
\usepackage{mathtools}
\usepackage{braket}
\usepackage{esint}
\newcommand{\abs}[1]{{\left|#1\right|}}
\newcommand{\norma}[1]{{\left\Vert#1\right\Vert}}

\usepackage{enumitem}
\usepackage{booktabs}
\usepackage{graphicx}
\usepackage{tikz}
\usetikzlibrary{patterns}
\usepackage{multicol}
\usepackage{caption}
\usepackage[skins,theorems]{tcolorbox}
\tcbset{highlight math style={enhanced,
colframe=black,colback=white,arc=0pt,boxrule=1pt}}
\def\XXint#1#2#3{{\setbox0=\hbox{$#1{#2#3}{\int}$}
    \vcenter{\hbox{$#2#3$}}\kern-.5\wd0}}

\theoremstyle{definition}
\newtheorem{definizione}{Definition}[section]
\theoremstyle{plain}
\newtheorem{theorem}{Theorem}[section]

\newtheorem{lemma}[theorem]{Lemma}
\newtheorem{prop}[theorem]{Proposition}
\newtheorem{corollario}[theorem]{Corollary}
\theoremstyle{definition}
\newtheorem{esempio}{Example}[section]
\newtheorem{oss}[esempio]{Remark}

\newtheorem*{open*}{Open problems}
\DeclareMathOperator{\R}{\mathbb{R}}
\newcommand{\OO}{\text{\normalfont  O}}

\makeatletter
\newcommand{\myfootnote}[2]{\begingroup
	\def\@makefnmark{}%
	\addtocounter{footnote}{-1}%
	\footnote{\textbf{#1} #2}
	\endgroup}
\makeatother

   \usepackage{hyperref}
  \hypersetup{linktoc=none, bookmarksnumbered, colorlinks=true, linkcolor=magenta, citecolor=cyan}

\title{On the stability of the annulus for the torsion of multiply connected domains}
\author{V. Amato, L. Barbato}
\date{}

\newcommand{\Addresses}{{% additional braces for segregating \footnotesize 
\bigskip 
  \footnotesize 
 \textit{E-mail address}, V. ~Amato (corresponding author): \texttt{v.amato@ssmeridionale.it} 

     \medskip 

 \noindent \textit{E-mail address}, L.~Barbato: \texttt{l.barbato@ssmeridionale.it} 
  
   \medskip 
 
  \textsc{Mathematical and Physical Sciences for Advanced Materials and Technologies, Scuola Superiore Meridionale, Largo San Marcellino 10, 80138 Napoli, Italy. }

 \par\nopagebreak 

}}

\begin{document}
\maketitle
\begin{abstract}
We establish a quantitative version of the isoperimetric inequality for the torsional rigidity of multiply connected domains, among sets with
given area and with given joint area of the holes.

Since the optimal shape is the annulus, we study how a domain approaches an annular configuration when its torsional rigidity is close to optimal. Our result shows that when the torsional rigidity is nearly optimal, the domain $\Omega$ must be close to an annulus.

\medskip

\end{abstract}
\section{Introduction}

The aim of this paper is to establish a stability result for the isoperimetric inequality associated with the torsional rigidity of multiply connected domains.

Let us consider a homogeneous beam whose cross-section is a multiply connected domain $\Omega$.

\begin{definizione}
    \label{multiplyconnected}
    A set $\Omega$ is a multiply connected domain if  $\Omega=G\setminus S$ and the following properties hold:
\begin{itemize}
    \item[\emph{(i)}] there exist $G,\Omega_1,\dots,\Omega_m$ bounded Lipschitz domain of $\R^n$; 
    \item[\emph{(ii)}] $\overline{\Omega}_i\subset G$ for $i=1,\dots,m$;
    \item[\emph{(iii)}]$\overline{\Omega}_i\cap\overline{\Omega}_j=\emptyset$ if $i\neq j$;
    \item[\emph{(iv)}]$S=\bigcup\limits_{i=1}^m\overline{\Omega}_i$.    
\end{itemize}
\end{definizione}

It is well known that the torsional rigidity of $\Omega$, namely that of a beam with cross-section $\Omega$, can be expressed as the maximum of a Rayleigh quotient (see, for instance, \cite{B,PW}), that is
\begin{equation}\label{torsione}
    T(\Omega)= \max\left\{ \frac{\displaystyle \left(\int_G \varphi\, dx\right)^2}{\displaystyle \int_G\abs{\nabla\varphi}^2\,dx} : \, \varphi \in H^1_0(G), \varphi\equiv \text{ constant on }  \Omega_i \, \forall i=1, \dots m  \right\}.
\end{equation}

In the case of simply connected domains ($m=0$), de Saint-Venant conjectured in 1856 that among all beams with cross-section of prescribed area, the one with circular cross-section has the largest torsional rigidity. A complete proof of this conjecture was given by Pólya in 1948, see \cite{P}.

The de Saint-Venant inequality was later generalized in several directions. For instance, in \cite{DW}, Díaz and Weinstein derived upper and lower bounds for the torsional rigidity of a beam, also in the case of multiply connected domains, in terms of its second-order moment.

In 1950, Pólya and Weinstein proved in \cite{PW} that among all beams with multiply connected cross-sections with given area and given total area of the holes, the annulus maximizes the torsional rigidity. Namely, using definition \eqref{torsione}, one has
\begin{equation}\label{polyaweinstein}
    T(\Omega) \leq T(\Omega^\OO),
\end{equation}
where $\Omega^\OO = G^\sharp \setminus S^\sharp$, and $G^\sharp$ and $S^\sharp$ denote concentric balls having the same measure as $G$ and $S$, respectively. This result has also been extended to any dimension and to certain degenerate elliptic operators in \cite{B,CH}.

Moreover, this problem can be viewed as a limiting case of a double-phase problem, which has been studied by several authors, see for example \cite{CLM,CMS,MNP}.

Once an isoperimetric inequality is established, a natural question concerns rigidity, namely, under which conditions the equality holds. In the simply connected case, this was proved by Pólya in \cite{P} and later revisited in \cite{BM}. For multiply connected domains, Pólya and Weinstein provided a partial answer: they showed that if equality holds in \eqref{polyaweinstein}, then the domain must have exactly one hole.

To the best of our knowledge, no sharper characterization of the equality case is available for multiply connected domains ($m \geq 1$). We prove the following result.

\begin{prop}\label{propmerc}
    Let $\Omega\subset\R^2$ satisfy Definition \ref{multiplyconnected}. If equality holds in \eqref{polyaweinstein}, that is
    \[T(\Omega)=T(\Omega^\OO),\]

     then, up to translation, $\Omega=\Omega^\OO$ and therefore there exists $x_0 \in \R^2$ such that $G=G^\sharp+ x_0$ and $S=S^\sharp+ x_0$.
\end{prop}

In Section \ref{section3}, we prove Proposition \ref{propmerc} as a particular case of a more general theorem valid in any dimension and for a broader class of linear elliptic operators.

Annuli are the only sets for which equality holds in \eqref{polyaweinstein}. This naturally leads to the question of stability.  More precisely, one would like to improve \eqref{polyaweinstein} by adding a remainder term on the left-hand side that quantifies how far a set is from the optimal annulus, that is the one with the same measure and the same hole measure.

To this end, we first introduce a way to measure the deviation from the corresponding annulus. We separately measure the distance of the outer boundary from a sphere and the distance of the holes from a ball. One possible tool is the Fraenkel asymmetry index; see, for instance, \cite{F} for its properties. This index is defined by
\begin{equation*}
	\alpha(\Omega) := \min_{x \in \mathbb{R}^{n}} \left\{ \frac{|\Omega \triangle B_r(x)|}{|B_r(x)|} \ : \ |B_r(x)| = |\Omega| \right\},
\end{equation*}
where $\triangle$ denotes the symmetric difference (see Subsection \ref{subsection2.3}).

We can now state our first main theorem.

\begin{theorem}\label{thm:dueasimmetrie}
     Let $\Omega\subset\R^2$ satisfy Definition \ref{multiplyconnected} such that $\displaystyle{\abs{S}\leq\frac{2}{3}\abs{G}}$. Then %there exists some constants $C_1=C_1(\abs{G})$, $C_2$ such that
    \begin{equation*}
        T(\Omega^\OO)-T(\Omega)\geq \frac{1}{3^22^{9} \pi \gamma_2}\left[\abs{G}^2  \alpha^3(G)+\abs{S}^2\alpha^3(S)\right].
     \end{equation*}
    where $\gamma_2$ denotes the constant appearing in the quantitative isoperimetric inequality in dimension two (see \cite{FMP2}).
\end{theorem}

The proof of this theorem relies on the propagation of asymmetry discussed in \cite{HN}, a technique that has been further developed and applied in several works, including \cite{AB,ABMP,ABCMP,BD,CLI,MS}.

We also point out that inequalities where annuli are optimal sets already appear in the literature, see \cite{Simon,CPP,GPPS,PPS,PPT}. For instance in \cite{CPP}, the authors studied a mixed eigenvalue problem on a domain with holes, providing an interpretation of the distance between $\Omega$ and the corresponding annulus. In that case as well, the proof is divided into two parts, dealing separately with the outer and inner boundaries.

Motivated by the classical Fraenkel asymmetry index, we introduce an \emph{annular asymmetry index} defined by
$$
\beta(\Omega) =\inf \{\abs{\Omega \triangle A} \, : \, A=B_1 \setminus B_2, \, \abs{B_1}= \abs{G} , \, \abs{B_2} = \abs{S}\},$$
which vanishes if and only if $\Omega$ is an annulus.

Using this notion, we state our second main result.
\begin{theorem}\label{thm:anellia}

 Let $\Omega\subset\R^2$ satisfy Definition \ref{multiplyconnected}. Then, there exists $C>0$ and $\overline{\theta}>0$, depending on $\abs{G}$ and $\abs{S}$, such that 
	\begin{equation*}
		%\label{teoremabeta}
		T(\Omega^\OO)-T(\Omega) \geq C \beta^{\overline{\theta}}(\Omega).
	\end{equation*}
\end{theorem}

The proof of this theorem relies on the closeness of the torsion function to the one of the optimal annulus. Since the holes correspond to regions where the gradient of the torsion function vanishes, this closeness implies not only that the outer domain and the holes are close to balls (as shown in Theorem \ref{thm:dueasimmetrie}), but also that these balls are almost concentric.

The paper is organized as follows. In Section \ref{section2}, we introduce notation and preliminary results. Section \ref{section3} is devoted to the proof of rigidity. In Section \ref{section4}, we show that the outer domain is close to a ball; in Section \ref{section5}, we prove that the asymmetry of the holes is small; and in Section \ref{section6}, we study the almost radiality of the torsion function. Finally, in Section \ref{section7}, we present a list of open problems.

\section{Notation and preliminary}
\label{section2}
In the following, we will denote by $\abs{\Omega}$ the Lebesgue measure of an open and bounded set of $\R^n$, with $n\geq2$. Also, the $ L^p-$ norm will be denoted as $\norma{\cdot}_p$ if the space is clear by the context; otherwise, the space will be explicitly written as $\norma{\cdot}_{L^p(\Omega)}$.
\begin{definizione}
    Let $\Omega$ be a bounded open set and let $ E\subset\Omega$. Then, the \textbf{perimeter} of $E$ inside $\Omega$, denoted by $P(E,\Omega)$, is defined as
   
    $$P(E,\Omega) = \sup\left\{ \int_E \text{div} \varphi \, dx : \varphi \in C_c^\infty(\Omega), \norma{\varphi}_{\infty} \leq 1\right\}.$$
    %$$
    %P(E,\Omega)= \inf \left\{\liminf_n \int_E \abs{\nabla u_n}\, dx : \, u \in \text{Lip}_{\text{loc}}(E)  , \, u_n \to \chi_E \in L^1_{\text{loc}}(E)\right\}.
    %$$
    If $\Omega = \R^n$, we denote $P(E):=P(E,\R^n)$.
\end{definizione}
Moreover, if $\Omega$ is an open set, the following coarea formula applies. Some references for results relative to the sets of finite perimeter and the coarea formula are, for instance, \cite{AFP,Fleming_Rishel,maggi2012sets}.
\begin{theorem}[Coarea formula]
	Let $\Omega \subset \mathbb{R}^n$ be an open set. Let $f\in W^{1,1}_{\text{loc}}(\Omega)$ and let $u:\Omega\to\R$ be a measurable function. Then,
	\begin{equation}
        	\label{coarea}
		{\displaystyle \int _{\Omega}u(x)|\nabla f(x)|dx=\int _{\mathbb {R} }dt\int_{\Omega\cap f^{-1}(t)}u(y)\, d\mathcal {H}^{n-1}(y)}.
	\end{equation}

Moreover, if  $f \in \text{BV}(\Omega)$, then, it holds the Fleming-Rishel formula, i.e.
\begin{equation}
	 \label{flemingrishel}
	\abs{Df}(\Omega) = \int_{-\infty}^{+\infty} P(\Set{ u>t},\Omega ) \, dt.
\end{equation}
\end{theorem}

The rest of this section is devoted to the introduction of the rearrangement of functions and both some qualitative and quantitative properties. 
\subsection{Rearrangement of functions}
Some general results and applications of the theory of rearrangements can be found in \cite{K}.

 \begin{definizione}\label{distribution:function}
	Let $u: \Omega \to \R$ be a measurable function. The \emph{distribution function} of $u$ is the function $\mu_u : [0,+\infty[\, \to [0, +\infty[$ defined as the measure of the superlevel sets of $u$, i.e.
	$$
	\mu_u(t)= \abs{\Set{x \in \Omega \, :\,  \abs{u(x)} > t}}.
	$$
\end{definizione}
Also using the Coarea formula \eqref{coarea}, it is possible to deduce an explicit expression for $\mu_u$ in terms of integrals of $u$

\begin{equation*}
    %\label{mugr0}
    \mu_u(t)=\abs{\{\abs{u}>t\}\cap \{|\nabla u |=0\}}+ \int_t^{+\infty} \left(\int_{u=s}\frac{1}{\abs{\nabla u}}\, d\, \mathcal{H}^{n-1}\right)\, ds.
\end{equation*}

As a consequence, for almost all $t\in (0,+\infty)$,
    \begin{equation}\label{brothers1}
        \infty>-\mu_u'(t)\geq\displaystyle \int_{u=t}\dfrac{1}{|\nabla u|}d\mathcal{H}^{n-1}.
    \end{equation}
Moreover, if $\mu_u$ is absolutely continuous, equality holds in \eqref{brothers1}.

\begin{definizione} \label{decreasing:rear}
	Let $u: \Omega \to \R$ be a measurable function, the \emph{decreasing rearrangement} of $u$, denoted by $u^\ast(\cdot)$, is defined as
$$u^*(s)=\inf\{t\geq 0:\mu_u(t)<s\}.$$
	\end{definizione}
 \begin{oss}
	Let us notice that the function $\mu_u(\cdot)$ is decreasing and right continuous and thus the function $u^\ast(\cdot)$ can be seen as its generalized inverse.
 \end{oss}
 Also by this remark, one can prove using Definitions \ref{distribution:function} and \ref{decreasing:rear} that
	 $$u^\ast (\mu_u(t)) \leq t, \quad \forall t\ge 0,$$ 
  $$\mu_u (u^\ast(s)) \leq s \quad \forall s \ge 0.$$ 
	\begin{definizione}
	 The \emph{Schwartz rearrangement} of $u$ is the function $u^\sharp $ whose superlevel sets are balls with the same measure as the superlevel sets of $u$. 
	\end{definizione}
		We emphasize the relation between $u^\sharp$ and $u^*$:
	$$u^\sharp (x)= u^*(\omega_n\abs{x}^n),$$
 where $\omega_n$ stands for the measure of the unit ball in $\R^n$.
  By their construction, the functions $u$, $u^*$ and $u^\sharp$ have the same distribution function, or they are equi-distributed, and it holds
$$ \displaystyle{\norma{u}_{p}=\norma{u^*}_{p}=\lVert{u^\sharp}\rVert_{p}}, \quad \text{for all } p\ge1.$$

The absolute continuity of $\mu_u$ is ensured by the following lemma (we refer, for instance, to \cite{BZ,CF2}).

\begin{lemma}

Let $u\in W^{1,p}(\R^n)$, with $p\in(1,+\infty)$. The distribution function $\mu_u$ of $u$ is absolutely continuous if and only if 
\begin{equation*}
    \abs{\{0<\abs{u}<  ||u^\sharp||_\infty\}\cap \{|\nabla u^\sharp| =0\}}=0.
\end{equation*}

\end{lemma}

The decreasing rearrangement has a lot properties such as the Hardy-Littlewood inequality, see \cite{HLP}. 

\begin{theorem}[Hardy-Littlewood inequaliy]
Let $\Omega $ be a bounded and open set of $\R^n$ and let us consider $h\in L^p(\Omega)$ and $g\in L^{p'}(\Omega)$. Then 
    \begin{equation*}
 \int_{\Omega} \abs{h(x)g(x)} \, dx \le \int_{0}^{\abs{\Omega}} h^*(s) g^*(s) \, ds.
\end{equation*}
\end{theorem}
Furthermore, if we assume that the function $u$ is a \emph{Sobolev function}, i.e. $u\in W^{1,p}_0(\Omega)$, then, also $u^\sharp$ is a Sobolev function and its gradient does not increase under symmetrization, as a consequence of the  P\'olya-Szegő inequality (see \cite{PS}).
\begin{theorem}[P\'olya-Szegő inequality]
Let $\Omega $ be a bounded and open set of $\R^n$ and 
     $u \in W^{1,p}_0(\Omega)$. Then, $u^{\sharp} \in W^{1,p}_0(\Omega^\sharp)$ and
	\begin{equation}\label{poliaszego}
		\lVert \nabla u^{\sharp} \rVert_{p} \leq \norma{\nabla u}_{p}.
	\end{equation}
\end{theorem}
Also, we recall the notion of pseudo-rearrangement of a function, introduced in \cite{AT}. Given $u$ measurable function in $\Omega$, $\forall s  \in [0, |\Omega|]$ we can find a subset $D_u(s)$ such that:
    \begin{enumerate}
        \item $|D_u(s)|=s$;
        \item  $s_1 \leq s_2
        \implies D_u(s_1)\subset D_u(s_2)$;
        \item $D_u(s)= \{x \in \Omega : \, |u(x)| >t\}$ if $s = \mu_u(t)$.
    \end{enumerate}
In other words, by the previous property, $D_u(s)$ coincides with a superlevel set whenever $u$ has no plateaus. When $u$ admits a plateau, $D_u(s)$ belongs to a continuous one-parameter family of subsets of $\{u \ge t\}$. In this latter case, such a one-parameter family is not unique.
    
Moreover, if $f\in L^p(\Omega)$ the function \[\int_{D_u(s)}f(x)\,dx\] is absolutely continuous, then there exists a function $F$ such that 
\begin{equation}\label{defpseudorear}
    \int_0^sF(t)\,dt=\int_{D_u(s)}f(x)\,dx.
\end{equation}
In particular, the following lemma \cite[Lemma 2.2]{AT} holds.
\begin{lemma}%\label{pseudorearr}
    Let $f\in L^p(\Omega),p>1$. There is a sequence $\{F_k(s)\}$ of functions which have the same rearrangement $f^*$ of $f$, such that \[F_k(s)\rightharpoonup F(s)\quad\text{in }L^p(0,|\Omega|).\]
    If $f\in L^1(\Omega)$, it follows that: \[\lim_k\int_0^{|\Omega|}F_k(s)g(s)\,ds=\int_0^{|\Omega|}F(s)g(s)\,ds\] for each function $g$ belonging to the space $BV(0,|\Omega|)$.
\end{lemma}
\subsection{Inequalities for the Torsional Rigidity}
% \label{2.2}
Let  us recall Buonocore's result in \cite{B}, that is a generalization for general degenerate elliptic operators in any dimension of \eqref{polyaweinstein}. To state it, we need the following definition.
\begin{definizione}
\label{torsion-buonocore}
    Let $\Omega$ satisfy Definition \ref{multiplyconnected} and let $\{a_{ij}(x)\}_{i,j=1,\dots,n}$ be a symmetric matrix, such that \begin{equation*}
    a_{ij}(x)\xi_i\xi_j\geq\nu(x)\abs{\xi}^2,\quad x\in\Omega,
    \end{equation*}
    with $\nu(x)\geq0,\nu\in L^1(\Omega),\nu^{-1}\in L^p(\Omega)$ for some $p>1$.

We define the quantity
\begin{equation*}
    T(\Omega,a_{ij})= \max\left\{ \frac{\displaystyle \left(\int_G \varphi\, dx\right)^2}{\displaystyle \int_Ga_{ij} \varphi_{x_i}\varphi_{x_j} \,dx} : \, \varphi \in H^1_0(G), \varphi\equiv \text{ constant on }  \Omega_i \, \forall i=1, \dots m  \right\}.
\end{equation*}
We will denote by $T(\Omega)=T(\Omega,\delta_{ij})$.

With this definition, the author proved in \cite[Theorem 2.1]{B} the following theorem.

\end{definizione}
\begin{theorem}\label{buonocore}

Let $\Omega$ satisfy Definition \ref{multiplyconnected} and let $\{a_{ij}(x)\}_{i,j=1,\dots,n}$ be a symmetric matrix, such that \begin{equation*}
    a_{ij}(x)\xi_i\xi_j\geq\nu(x)\abs{\xi}^2,\quad x\in\Omega
    \end{equation*}
    with $\nu(x)\geq0,\nu\in L^1(\Omega),\nu^{-1}\in L^p(\Omega)$ for some $p>1$ and $T(\Omega,a_{ij})$ as in Definition \ref{torsion-buonocore}.
    
If $\Omega^\OO 
=G^\sharp \setminus S^\sharp $, then
$$T(\Omega,a_{ij})\leq T(\Omega^\OO,\nu^\ast(\omega_{n}\abs{x}^n- \abs{S})\delta_{ij}).$$
\end{theorem}

\begin{oss}
The Euler-Lagrange equation of the functional $T(\Omega,a_{ij})$, as shown in \cite{B,PW}, see also \cite{BS,K}, is given by \begin{equation}\label{torsionprob}
        \begin{cases}
            -\left(a_{ij}z_{x_i}\right)_{x_j}=1 & \text{ in } \Omega \\
            z = 0 & \text{on } \partial G \\
            z=c_i & \text{on } \partial \Omega_i
        \\ \displaystyle{\int_{\partial \Omega_i} \frac{\partial z}{\partial  \nu } \, d\sigma =\abs{\Omega_i} }& \forall i=1,\dots m.\end{cases}
    \end{equation}  
    
    Hence, if $\tilde u$ is a maximizer of $T(\Omega,a_{ij})$, then its restriction to $\Omega$ is solution to \eqref{torsionprob}.
    
    The constants $(c_1,\dots,c_m)\in\R^m$ for which problem \eqref{torsionprob} admits a unique solution are unknown and are determined a posteriori.

\end{oss}

\begin{oss}
    For the sake of readability, we point out that the unique solution to the problem 
     \begin{equation}
    	\label{probsimm}
        \begin{cases}
            \displaystyle{-\left(\nu^*\left(\omega_n\abs{x}^n-\abs{S}\right)V_{x_i}\right)_{x_i}=1} & \text{ in } \Omega^\OO \\
            \displaystyle{V = 0} & \text{on } \partial G^\sharp \\
            \displaystyle{V=n^{-2}\omega_n^{-\frac{2}{n}}\int_{\abs{S}}^\abs{G}r^{-1+\frac{2}{n}}\left(\nu^*\left(r-\abs{S}\right)\right)^{-1}\,dr} & \text{on } \partial S^\sharp.
        \end{cases}
    \end{equation}
    is given by 
    \begin{equation}\label{defV}
      V(x)=n^{-2}\omega_n^{-\frac{2}{n}}\int_{\omega_n\abs{x}^n}^\abs{G}r^{-1+\frac{2}{n}}\left(\nu^*\left(r-\abs{S}\right)\right)^{-1}\,dr.
    \end{equation}

    Thus, $V$ is precisely the torsion function of the annulus, and throughout the remainder of the paper we will refer to $V$ as the torsion function of the annulus.
\end{oss}

 We define the pseudo-rearrangement of $\nu^{-1}$ on the superlevel sets of $u$ as the function $\hat \nu^{-1}$ for which
\begin{equation}
    \label{hatnu}
    \int_0^s \hat \nu^{-1}\, dx = \int_{D_u(s)} \nu^{-1}\, dx, \forall s \in (0 , \abs{\Omega}).
\end{equation}

We recall the following  approximation result for $\hat\nu^{-1}$.
\begin{lemma}\label{approx}
    There is a sequence $\{\nu_k^{-1}\}$ such that $(\nu_k^{-1})^*=(\nu^{-1})^*$ (or, equivalently, $\nu_k^\ast=\nu^\ast$ ) and $\{\nu^{-1}_k(s)\}$ is weakly convergent in $L^p(0,\abs{\Omega})$ to $\displaystyle \hat\nu^{-1}$, defined in \eqref{hatnu}.
\end{lemma}
Also, it is useful to recall the definition of the function $\zeta^{-1}$ given in \cite{B} which allows to explicitly write a solution to \eqref{torsionprob}. Let the constants $c_h$ be ordered according to their magnitude
\[0=c_0<c_1<\dots< c_m,
\]
we set \[I_h=\left(\mu_u(c_h)+\sum_{k=h+1}^m\abs{\Omega_k},\mu_u(c_{h})+\sum_{k=h}^m\abs{\Omega_k}\right),\quad h=1,\dots,m-1\]
and 
\[ J_h=\left(\mu_u(c_h)+\sum_{k=h}^m\abs{\Omega_k},\mu_u(c_{h-1})+\sum_{k=h}^m\abs{\Omega_k}\right),\quad h=1,\dots,m.\]
Using notation of Lemma \ref{approx} we define
\begin{equation}
\label{zetak}\zeta^{-1}_k(r)=\begin{cases}
    \displaystyle{\nu_k^{-1}(r)} & \text{ if } r\in\left(0,\mu_u(c_m)\right)\\
    \displaystyle{0} & \text{ if } r\in\left(\mu_u(c_m),\mu_u(c_m)+|\Omega_m|\right) \quad \text{or }r\in I_h\\
    \displaystyle{\nu_k^{-1}\left(r-\sum_{k=h}^m\abs{\Omega_k}\right)} & \text{ if } r\in J_h
\end{cases}, \end{equation}
which, again by Lemma \ref{approx}, weakly converges in $L^p$ to
\begin{equation} \label{zeta}
    \zeta^{-1}(r)=\begin{cases}
    \displaystyle{\hat\nu^{-1}(r)} & \text{ if } r\in\left(0,\mu_u(c_m)\right)\\
    \displaystyle{0} & \text{ if } r\in\left(\mu_u(c_m),\mu_u(c_m)+|\Omega_m|\right)\quad \text{or }r\in I_h\\
    \displaystyle{\hat\nu^{-1}\left(r-\sum_{k=h}^m\abs{\Omega_k}\right)} & \text{ if } r\in J_h
\end{cases}.
\end{equation}

Moreover, we use now and in what follows, the tilde notation indicates the extension of a function by constants in the holes.

In order to as clear as possible, we now sketch the proof of Theorem \ref{buonocore} proved in \cite{B}.

\begin{proof}[Proof of Theorem \ref{buonocore}]
    We apply the isoperimetric inequality to the level set $\{\tilde{u}>t\}$ and the Fleming--Rishel formula \eqref{flemingrishel},  getting 
\begin{equation}\label{buonocoreisoperimetrica1}
    \begin{aligned}
      n^2\omega_n^{\frac{2}{n}}\mu_{\tilde u}^{2-\frac{2}{n}}(t) \leq  P^2(\tilde u>t)&=\left(-\frac{d}{dt}\int_t^\infty P(\tilde u>s)\,ds\right)^2=\left(-\frac{d}{dt}\int_{\tilde u>t}\abs{\nabla u}\,dx\right)^2\\&\leq\left(-\frac{d}{dt}\int_{\tilde u>t}\nu^{-1}(x)\,dx\right)\left(-\frac{d}{dt}\int_{\tilde u>t}\nu(x)\abs{\nabla u}^2\, dx\right)\\&\leq\left(-\frac{d}{dt}\int_{\tilde u>t}\nu^{-1}(x)\,dx\right)\left(-\frac{d}{dt}\int_{\tilde u>t}a_{i,j} u_{x_i}u_{x_j}\, dx\right)\\&= -\mu_{\tilde u}'(t) \zeta^{-1}[\mu_{\tilde u}(t)] \mu_{\tilde u}(t).
    \end{aligned}
    \end{equation}
  
    If we divide by $n^2\omega_n^{\frac{2}{n} }\mu_{\tilde u}^{2-\frac{2}{n}}(t)$ and we integrate between $0$ and $\tau$, after a change of variable, we get
    \begin{equation*}
            \tau\leq n^{-2}\omega_n^{-\frac{2}{n}}\int_{\mu_{\tilde u}(\tau)}^{\abs{G}}r^{-1+\frac{2}{n}}\zeta^{-1}(r)\,dr.
    \end{equation*}
    
The last, due to the definition of decreasing rearrangement, can be written as  \begin{equation}\label{puntuale}
      \Tilde{u}^*(s)\leq h(s)=n^{-2}\omega_n^{-\frac{2}{n}}\int_s^\abs{G}r^{-1+\frac{2}{n}}\zeta^{-1}(r)\,dr,\quad s\in(0,\abs{G}).
  \end{equation}
  Lastly, by applying the Hardy-Littlewood inequality, Lemma \ref{approx} and the functions defined in \eqref{zetak}, the thesis can be obtained through the following chain of inequalities
      \begin{equation}\label{saintven}
    \begin{aligned}
        T(\Omega,a_{ij})&\leq\int_0^{\abs{G}}h(s)\, ds \\&=\frac{1}{n^2\omega_n^{\frac{2}{n}}}\int_0^{\abs{G}}\int_s^{\abs{G}}r^{-1+\frac{2}{n}}\zeta^{-1}(r)\,dr\,ds\\&=\frac{1}{n^2\omega_n^{\frac{2}{n}}}\int_0^{\abs{G}}s^\frac{2}{n}\zeta^{-1}(s)\,ds   
        \\&= \lim_k\frac{1}{n^{2}\omega_n^{\frac{2}{n}}}\int_{0}^{\abs{G}}s^{\frac{2}{n}}\zeta_k^{-1}(s)\,ds\\
        &\leq\frac{1}{n^2\omega_n^{\frac{2}{n}}}\int_{\abs{S}}^{\abs{G}}s^\frac{2}{n}(\nu^*(s-\abs{S}))^{-1}\,ds\\&=T(\Omega^\OO,\nu^\ast(\omega_n \abs{x}^n -\abs{S})\delta_{ij}).
    \end{aligned}
    \end{equation}
\end{proof}
 
\begin{corollario}\label{puntualedim2}Let $n=2$, and let $u$ and $V$ be the torsion functions associated with 
$T(\Omega, a_{ij})$ and $T\!\left(\Omega^\OO, \nu^\ast(\omega_n |x|^n - |S|)\delta_{ij}\right)$, respectively.  
Let $\tilde u$ and $\tilde V$ be the extensions to $G$ and $G^\sharp$, respectively. Then
\[
\tilde u^\sharp(x) \leq \tilde V(x), \quad x \in G^\sharp.
\]
\end{corollario}

Eventually, the result \ref{puntualedim2} can't be extended to higher dimensions as shown in the following remark.
\begin{oss}
    Let $n\geq2$, consider $G=B_1$ and $S=B_R\setminus B_r$ with $r\leq R$.
    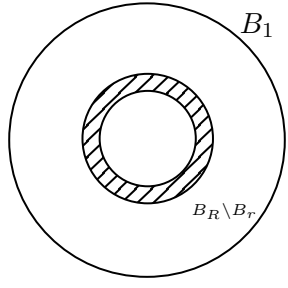
\begin{figure}[!ht]
    \centering
    % Pattern Info
     
    \tikzset{
    pattern size/.store in=\mcSize, 
    pattern size = 5pt,
    pattern thickness/.store in=\mcThickness, 
    pattern thickness = 0.3pt,
    pattern radius/.store in=\mcRadius, 
    pattern radius = 1pt}
    \makeatletter
    \pgfutil@ifundefined{pgf@pattern@name@_8dszkp3t7}{
    \pgfdeclarepatternformonly[\mcThickness,\mcSize]{_8dszkp3t7}
    {\pgfqpoint{0pt}{0pt}}
    {\pgfpoint{\mcSize+\mcThickness}{\mcSize+\mcThickness}}
    {\pgfpoint{\mcSize}{\mcSize}}
    {
    \pgfsetcolor{\tikz@pattern@color}
    \pgfsetlinewidth{\mcThickness}
    \pgfpathmoveto{\pgfqpoint{0pt}{0pt}}
    \pgfpathlineto{\pgfpoint{\mcSize+\mcThickness}{\mcSize+\mcThickness}}
    \pgfusepath{stroke}
    }}
    \makeatother
    \tikzset{every picture/.style={line width=0.75pt}} %set default line width to 0.75pt        
    
    \begin{tikzpicture}[x=0.72pt,y=0.72pt,yscale=-1,xscale=1]
    
    %Shape: Circle [id:dp3129074480521621] 
    \draw  [pattern=_8dszkp3t7,pattern size=6pt,pattern thickness=0.75pt,pattern radius=0pt, pattern color={rgb, 255:red, 0; green, 0; blue, 0}] (309.47,141.86) .. controls (309.47,123.12) and (324.66,107.93) .. (343.39,107.93) .. controls (362.13,107.93) and (377.32,123.12) .. (377.32,141.86) .. controls (377.32,160.6) and (362.13,175.79) .. (343.39,175.79) .. controls (324.66,175.79) and (309.47,160.6) .. (309.47,141.86) -- cycle ;
    %Shape: Circle [id:dp0867496229818816] 
    \draw   (271.37,142.65) .. controls (271.37,103.08) and (303.45,71) .. (343.02,71) .. controls (382.59,71) and (414.67,103.08) .. (414.67,142.65) .. controls (414.67,182.22) and (382.59,214.3) .. (343.02,214.3) .. controls (303.45,214.3) and (271.37,182.22) .. (271.37,142.65) -- cycle ;
    %Shape: Circle [id:dp5401756101947408] 
    \draw  [fill={rgb, 255:red, 255; green, 255; blue, 255 }  ,fill opacity=1 ] (318.39,141.86) .. controls (318.39,128.05) and (329.59,116.86) .. (343.39,116.86) .. controls (357.2,116.86) and (368.39,128.05) .. (368.39,141.86) .. controls (368.39,155.67) and (357.2,166.86) .. (343.39,166.86) .. controls (329.59,166.86) and (318.39,155.67) .. (318.39,141.86) -- cycle ;
    
    % Text Node
    \draw (389.76,74.42) node [anchor=north west][inner sep=0.75pt]    {$B_{1}$};
    % Text Node
    \draw (364.56,174.28) node [anchor=north west][inner sep=0.75pt]  [font=\tiny]  {$B_{R} \backslash B_{r}$};

    \end{tikzpicture}
     \caption{$\Omega=B_1\setminus S.$}
 \end{figure}   
    \newline Given that the domain is already radially symmetric by \eqref{puntuale} we have an explicit expression of the solution:
\[u(s)=n^{-2}\omega_n^{-\frac{2}{n}}\int_s^{\abs{G}}t^{-1+\frac{2}{n}}\zeta^{-1}(t)\,dt,\qquad s\in(0,\abs{G}).\] As $\zeta^{-1}$ has an explicit expression \eqref{zeta}, we can distinguish three cases:
\begin{itemize}
    \item $s\geq \abs{B_R}$ \[u(s)=n^{-2}\omega_n^{-\frac{2}{n}}\int_s^{\abs{G}}t^{-1+\frac{2}{n}}\,dt=\frac{n^{-1}\omega_n^{-\frac{2}{n}}}{2}\left(\abs{G}^\frac{2}{n}-s^\frac{2}{n}\right);\]
    \item $\abs{B_r}<s<\abs{B_R}$ \[u(s)=\frac{n^{-1}\omega_n^{-\frac{2}{n}}}{2}\left(\abs{G}^\frac{2}{n}-\abs{B_R}^\frac{2}{n}\right);\]
    \item $s< \abs{B_r}$ \[u(s)=\frac{n^{-1}\omega_n^{-\frac{2}{n}}}{2}\left(\abs{G}^\frac{2}{n}-\abs{B_R}^\frac{2}{n}\right)+\int_s^\abs{B_r}t^{-1+\frac{2}{n}}\,dt=\frac{n^{-1}\omega_n^{-\frac{2}{n}}}{2}\left(\abs{G}^\frac{2}{n}-\abs{B_R}^\frac{2}{n}+\abs{B_r}^\frac{2}{n}-s^\frac{2}{n}\right).\]
\end{itemize}
On the other hand, the solution to the symmetrized problem is:
\begin{equation*}
    V(s)=\begin{cases}
        \displaystyle{\frac{n^{-1}\omega_n^{-\frac{2}{n}}}{2}(\abs{G}^\frac{2}{n}-s^\frac{2}{n})} & \text{if } s\geq\abs{S}\\
        \displaystyle{\frac{n^{-1}\omega_n^{-\frac{2}{n}}}{2}(\abs{G}^\frac{2}{n}-\abs{S}^\frac{2}{n})} & \text{if } s<\abs{S}.
    \end{cases}
\end{equation*}
Therefore, by noticing that $\abs{B_R}=\abs{S}+\abs{B_r}$ we get
\begin{equation*}
    V(0)-u(0)=\frac{n^{-1}\omega_n^{-\frac{2}{n}}}{2}\left(\abs{B_R}^\frac{2}{n}-\abs{B_r}^\frac{2}{n}-\abs{S}^\frac{2}{n}\right)=\frac{n^{-1}\omega_n^{-\frac{2}{n}}}{2}\left(\left(\abs{B_r}+\abs{S}\right)^\frac{2}{n}-\abs{B_r}^\frac{2}{n}-\abs{S}^\frac{2}{n}\right),
\end{equation*}
which is not positive by concavity of $t \to t^{2/n}$ whenever $n\geq3$. Therefore, in dimension $n>2$, a point-wise comparison does not hold and, moreover, can't hold a mass concentration comparison of the type
\begin{equation}\label{concentrazioni}
    \int_0^r    \tilde u^\ast(s)\,ds\leq\int_0^r\tilde V^\ast(s)\,ds\quad\forall\, r\in[0,\abs{G}],
\end{equation}
because \eqref{concentrazioni} would implicate $L^p$ comparison for every $p$ (see \cite{ALT}), including $p=\infty$.
\end{oss}
\subsection{Quantitative inequalities}
\label{subsection2.3}
The technique to prove our main theorem is to apply some quantitative inequalities and subsequently estimate how the distance from the optimum propagates. Denoting by $\alpha(\Omega)$ the Fraenkel's asymmetry index \begin{equation*}
	\alpha(\Omega):=\min_{x \in \R^{n}}\bigg \{  \dfrac{|\Omega\triangle B_r(x)|}{|B_r(x)|} \;,\; |B_r(x)|=|\Omega|\bigg \} \text{ where } |\Omega\triangle B_r(x)|= |(\Omega\setminus B_r(x)) \cup ( B_r(x)\setminus\Omega)|,
\end{equation*}
 we have the following quantitative isoperimetric inequality, proved in \cite{FMP2} (see also \cite{CL,Fuglede,FMP,H}).

\begin{theorem}
    There exists a constant $\gamma_n$ such that,  for any measurable set $\Omega$ of finite measure
    \begin{equation}\label{quant_isop}
        P(\Omega)\geq n\omega_n^{\frac{1}{n}}\abs{\Omega}^{\frac{n-1}{n}}\left(1+\dfrac{\alpha^2(\Omega)}{\gamma_n}\right),
    \end{equation}
where
$$\gamma_n= \frac{181\, n^7}{(2-2^{(n-1)/n})^{\frac{3}{2}}}.$$
\end{theorem}

It is not trivial to understand how the asymmetry propagates from the whole domain to the superlevel sets of the solution. Nevertheless, we can use the following result (see, for instance,~\cite[Lemma~2.8]{BD}), which provides an estimate for sets that are close in measure.

\begin{lemma}\label{propasi}
    Let $\Omega\subset\R^n$ be an open set with finite measure and $U\subset\Omega$ a subset of positive measure such that \[\frac{|\Omega\backslash U|}{|\Omega|}\leq\frac{1}{4}\alpha(\Omega).\]
    Then, there holds \[\alpha(U)\geq\frac{1}{2}\alpha(\Omega).\]
\end{lemma}
Another key ingredient is the quantitative version of P\'olya-Szeg\H o
inequality \eqref{poliaszego}, proved in \cite{CFET}.
\begin{theorem}[Quantitative P\'olya-Szeg\H o
inequality]
\label{polya_quant}
    Let $u\in W^{1,2}(\R^n)$, $n\geq 2$. Then, there exist  positive constants $r$, $s$ and $C$, depending only on $n$,  such that, for every $u\in W^{1,2}(\R^n)$, it holds
\begin{equation*}
   \inf_{x_0 \in \R^n} \dfrac{\displaystyle{\int_{\R^n} \abs{u(x)\pm u^\sharp(x+x_0)} \, dx }}{\abs{\{\abs{u}>0\}}^{\frac{1}{n}+\frac{1}{2}}\norma{\nabla u^\sharp}_2}\leq C(n) \left[M_{u^\sharp}(E(u)^r)+E(u)\right]^s,
\end{equation*}
where 
\begin{equation}\label{eumu}
E(u)= \frac{\displaystyle{\int_{\R^n} \abs{\nabla u}^2}}{\displaystyle{\int_{\R^n} |\nabla u^{\sharp}|^2}}-1 \qquad \text{ and } \qquad M_{u^\sharp}(\delta)=\dfrac{\abs{\left\{|\nabla u^{\sharp}|<\delta\right\}\cap \left\{0<u^\sharp<||u||_\infty\right\}}}{\abs{\{\abs{u}>0\}}}.
\end{equation}
\end{theorem}
Lastly we recall, also, the sharp quantitative inequality for the torsional rigidity (see \cite{BDV}) in the case of homogeneous Dirichlet boundary condition
\begin{theorem}
    There exists a constant $C$ depending only on the dimension $n$, such that for every open set $\Omega\subset\R^n$ with finite measure we have
    \[\abs{B}^{-\frac{n+2}{n}}T(B)-\abs{\Omega}^{-\frac{n+2}{n}}T(\Omega)\geq C\alpha^2(\Omega).\]
\end{theorem}
\section{ Rigidity of the problem}\label{section3}
As already mentioned in the introduction, we prove Proposition \ref{propmerc} in a more general setting. Let $u$ and $V$ be solutions to \eqref{torsionprob} and \eqref{probsimm} respectively. Then, we have

\begin{theorem}\label{teorig}
Let $\Omega$ satisfy Definition \ref{multiplyconnected}. Let $n\geq 2$,  $\nu\in L^1(\Omega),\nu^{-1}\in L^p(\Omega)$ for some $p>1$, and u and v  be  solutions to \eqref{torsionprob} and \eqref{probsimm} respectively. Then, if 
 $$T(\Omega,a_{ij})=T(\Omega^\OO,\nu^*(\omega_n\abs{x}^n-\abs{S})\delta_{ij}),$$
 necessarily $\Omega= \Omega^\OO+ x_0$, $u=u^\sharp(\cdot +x_0)$, $\nu=\nu^\sharp(\cdot +x_0)$  and $a_{ij}(x+x_0)x_j=\nu(x)x_i$ for some $x_0 \in \R^n$.
\end{theorem}

It is clear that Proposition \ref{propmerc} is a particular case when $n=2$ and $a_{ij}= \delta_{ij}$(hence $\nu(x)\equiv 1$).

Before proving rigidity, we recall  the Brothers and Ziemer theorem contained in \cite{BZ}.

\begin{theorem}\label{brotherziemer}
    Let $u$ be a nonnegative function in $W_0^{1,p}(\Omega)$ with $1<p<\infty$, such that
    \[\abs{\{x:\nabla u^\sharp(x)=0\}\cap u^{\sharp^{-1}}(0,\norma{u}_\infty)}=0.\]
 If the following equality holds:
    \[\int_\Omega\abs{\nabla u}^p\,dx=\int_{\Omega^\sharp}\abs{\nabla u^\sharp}^p\,dx,\]
    then, there exists $x_0 \in \R^n$ such that $$\Omega= \Omega^\sharp+ x_0, \qquad \text{and} \qquad u=\pm u^\sharp(\cdot +x_0).$$
\end{theorem}
Now we are able to prove the rigidity result.
\begin{proof}[Proof of Theorem \ref{teorig}]
We start by recalling that
\[
T(\Omega,a_{ij})=\int_G \tilde u\,dx,
\qquad 
T\!\left(\Omega^\OO,\nu^*(\omega_n|x|^n-|S|)\delta_{ij}\right)
=\int_{G^\sharp}\tilde V\,dx.
\]

We introduce the extension at zero of $(\nu^\ast)^{-1}$:
\[
\overline{\nu}^{-1}(s)=
\begin{cases}
0 & \text{if } s\in(0,|S|),\\[4pt]
\bigl(\nu^*(s-|S|)\bigr)^{-1} & \text{if } s\in(|S|,|G|).
\end{cases}
\]

If equality holds, then all terms in \eqref{saintven} are equal. In particular,
\begin{equation*}%\label{uguhardylittlewood}
\int_0^{|G|} s^{\frac{2}{n}}\zeta^{-1}(s)\,ds
=
\int_0^{|G|} s^{\frac{2}{n}}\overline{\nu}^{-1}(s)\,ds.
\end{equation*}

Since the function $s \mapsto s^{2/n}$ is nonnegative, strictly increasing, and absolutely continuous, the equality case in the Hardy--Littlewood inequality (see \cite{BM}) implies
\begin{equation}\label{zeta=nu}
\zeta^{-1}(s)=\overline{\nu}^{-1}(s)
\quad \text{for a.e. } s\in(0,|G|),
\end{equation}
and $\nu^{-1}$ has the same superlevel sets as $|u|$.

Substituting \eqref{zeta=nu} into the definition of $h$ and using \eqref{puntuale}, we obtain
\[
\tilde u^*(s)\le h(s)=\tilde V^*(s)
\quad \text{for all } s\in(0,|G|).
\]

Moreover,
\[
\int_G \tilde u\,dx
=
\int_{G^\sharp} \tilde u^\sharp\,dx
=
T(\Omega,a_{ij})
=
T\!\left(\Omega^\OO,\nu^*(\omega_n|x|^n-|S|)\delta_{ij}\right)
=
\int_{G^\sharp}\tilde V\,dx,
\]
hence
\begin{equation}\label{uast=V}
\tilde u^\sharp=\tilde V,
\end{equation}
which implies that the gradient vanishes only on the top level set.

\medskip

Equality in \eqref{saintven} also implies equality in \eqref{buonocoreisoperimetrica1}, namely
\[
-\frac{d}{dt}\int_{\{|u|>t\}}\nu|\nabla u|^2\,dx
=
-\frac{d}{dt}\int_{\{|u|>t\}} a_{ij}u_{x_i}u_{x_j}\,dx.
\]
Together with
\[
\int_{\{|u|>\|u\|_\infty\}}\nu|\nabla u|^2\,dx
=
\int_{\{|u|>\|u\|_\infty\}}a_{ij}u_{x_i}u_{x_j}\,dx
=0,
\]
this yields
\begin{equation*}%\label{equalitytorsioni}
\int_\Omega \nu|\nabla u|^2\,dx
=
\int_\Omega a_{ij}u_{x_i}u_{x_j}\,dx
=
T(\Omega,a_{ij}).
\end{equation*}
Using \eqref{uast=V}, we obtain
\begin{equation}\label{equalityint}
\int_{\Omega^\OO}
\nu^*(\omega_n|x|^n-|S|)
|\nabla \tilde u^\sharp|^2\,dx
=
\int_\Omega \nu(x)|\nabla \tilde u|^2\,dx.
\end{equation}
From \eqref{equalityint} and coarea formula it follows that
\begin{equation}\label{bmm}
\int_0^{+\infty} 
\!\!\left(\int_{\{\nu^\ast >t\} \cup S^{\sharp}} 
|\nabla \tilde u^\sharp|^2 \, dx \right) dt
=
\int_0^{+\infty} 
\!\!\left(\int_{\{\nu >t\} \cup S} 
|\nabla \tilde u|^2 \, dx \right) dt.
\end{equation}
Arguing as in \cite{BM}, since $\nu^{-1}$ has the same superlevel sets as $|u|$, one can prove that
\[
\int_{\{\nu^\ast >t\} \cup S^{\sharp}} 
|\nabla \tilde u^\sharp|^2 \, dx
\le
\int_{\{\nu >t\} \cup S} 
|\nabla \tilde u|^2 \, dx
\quad \forall t \ge 0.
\]
Combining with \eqref{bmm}, we deduce equality for all $t$, and in particular
\[
\int_{G^\sharp} |\nabla \tilde u^\sharp|^2 \, dx
=
\int_G |\nabla \tilde u|^2 \, dx.
\]
By Theorem \ref{brotherziemer}, up to translations,
\[
G = G^\sharp,
\qquad
\tilde u = \pm \tilde u^\sharp.
\]
Since the holes correspond to the maximum level sets of $\tilde u$, it follows that
\[
S = S^\sharp.
\]
Moreover, by \eqref{zeta=nu}, $\nu^{-1}$ has the same superlevel sets as $|u|$, hence
\[
\nu(x)=\nu^\ast(\omega_n |x|^n - |S|).
\]
Finally, \eqref{equalityint} becomes
\[
\int_{\Omega^\OO} a_{ij}(x)V_{x_i}V_{x_j}\,dx
=
\int_{\Omega^\OO}
\nu^\ast(\omega_n|x|^n-|S|)|\nabla V|^2\,dx.
\]
Since
\[
a_{ij}(x)\xi_i\xi_j
\ge
\nu^\ast(\omega_n|x|^n-|S|)|\xi|^2,
\quad
\xi_i = V_{x_i} = -x_i\,n^{-1}\nu^\ast(\omega_n|x|^n-|S|)^{-1},
\]
we obtain
\[
\sum_{i,j=1}^n\frac{a_{ij}(x)}{\nu^\ast(\omega_n|x|^n-|S|)}x_ix_j
=
|x|^2
\quad \text{a.e. in } \Omega.
\]
Since the matrix
\[
\frac{a_{ij}(x)}{\nu^\ast(\omega_n|x|^n-|S|)}
\]
is symmetric and its smallest eigenvalue is $\ge 1$, the above identity implies that the first eigenvalue is exactly $1$ and $x$ is a corresponding eigenvector. This concludes the proof.
\end{proof}
\section{Asymmetry of the exterior domain}\label{section4}

In this section we show that the Fraenkel asymmetry index for the exterior domain $G$ can be controlled by the deficit in the P\'olya-Weinstein inequality. 

 To this aim we recall the definition of $s_G$, given in \cite{BD}:
    \begin{equation*}
        s_{G}=\sup\left\{ t\geq 0: \mu_{\tilde u}(t)\geq \abs{G}\left(1-\dfrac{\alpha(G)}{4}\right)\right\} \in \R.
    \end{equation*}
We recall that the tilde notation is used to denote the extension of a function by constant values on the holes of the domain.
Then, the following statement holds.
\begin{lemma}\label{4.1}
Let $\Omega\subset\R^n$ satisfy Definition \ref{multiplyconnected} then, it holds
    \begin{equation}
        \label{essamitorsion}
         T(\Omega^\OO)- T(\Omega)\geq s_G\abs{G}\frac{\alpha(G)^2}{4\gamma_n},
    \end{equation}

    where $\gamma_n$ is the constant appearing in the quantitative isoperimetric inequality \eqref{quant_isop}.
\end{lemma}
\begin{proof}
We assume that $\alpha(G)>0$, since otherwise the statement is trivial. 
Applying the quantitative isoperimetric inequality \eqref{quant_isop} to the superlevel set $\{\tilde u > t\}$, we obtain
\begin{equation*}
n^2\omega_n^{\frac{2}{n}}\mu_{\tilde u}^{2-\frac{2}{n}}(t)\left(1+\frac{2}{\gamma_n}\alpha^2(\tilde u>t)\right)
\leq 
n^2\omega_n^{\frac{2}{n}}\mu_{\tilde u}^{2-\frac{2}{n}}(t)\left(1+\frac{1}{\gamma_n}\alpha^2(\tilde u>t)\right)^2
\leq P^2(\tilde u>t).
\end{equation*}

Arguing as in the proof of Theorem \ref{buonocore}  for the right-hand side, that is and using the Fleming--Rishel formula \eqref{flemingrishel} together with H\"older's inequality, we deduce
\begin{equation*}
n^2\omega_n^{\frac{2}{n}}\mu_{\tilde u}^{2-\frac{2}{n}}(t)
\left(1+\frac{2}{\gamma_n}\alpha^2(\tilde u>t)\right)
\leq -\mu'_{\tilde u}(t)\,\zeta^{-1}\!\bigl(\mu_{\tilde u}(t)\bigr)\,\mu_{\tilde u}(t).
\end{equation*}
Dividing by $n^2\omega_n^{\frac{2}{n}}\mu_{\tilde u}^{2-\frac{2}{n}}(t)$ and integrating between $0$ and $\tau$, after a change of variables we obtain
\begin{equation*}
\tau + \frac{2}{\gamma_n}\int_0^\tau \alpha^2(\tilde u>t)\,dt
\leq n^{-2}\omega_n^{-\frac{2}{n}}
\int_{\mu_{\tilde u}(\tau)}^{|G|} r^{-1+\frac{2}{n}}\zeta^{-1}(r)\,dr.
\end{equation*}
By the definition of the decreasing rearrangement, this can be rewritten as
\begin{equation*}
\frac{2}{\gamma_n}\int_0^{\tilde u^*(r)} \alpha^2(\tilde u>t)\,dt
\leq n^{-2}\omega_n^{-\frac{2}{n}}
\int_r^{|G|} t^{-1+\frac{2}{n}}\zeta^{-1}(t)\,dt - \tilde u^*(r)
:= h(r)-\tilde u^*(r).
\end{equation*}
Integrating over $[0,|G|]$, we obtain
\begin{equation*}
\begin{aligned}
\frac{2}{\gamma_n}\int_0^{|G|}\int_0^{\tilde u^*(r)} \alpha^2(\tilde u>t)\,dt\,dr
&\leq n^{-2}\omega_n^{-\frac{2}{n}}
\int_0^{|G|}\int_r^{|G|} t^{-1+\frac{2}{n}}\zeta^{-1}(t)\,dt\,dr
- \int_0^{|G|}\tilde u^*(r)\,dr \\
&= \int_0^{|G|} h(r)\,dr - T(\Omega).
\end{aligned}
\end{equation*}
Therefore, by \eqref{saintven}, it follows that
\begin{equation}\label{common41e54}
\frac{2}{\gamma_n}\int_0^{|G|}\int_0^{\tilde u^*(r)} \alpha^2(\tilde u>t)\,dt\,dr
\leq T(\Omega^{\mathrm{O}}) - T(\Omega).
\end{equation}
We now define
\begin{equation*}
A := \left\{ t \geq 0 : \mu_{\tilde u}(t) \geq |G|\left(1-\frac{\alpha(G)}{4}\right) \right\}.
\end{equation*}
Then $0\in A$ whenever $\alpha(G)>0$, and $A$ is an interval since $\mu_{\tilde u}$ is decreasing. Moreover, for every $t\in A$,
\begin{equation*}
\frac{|G\setminus\{\tilde u>t\}|}{|G|}
= 1 - \frac{\mu_{\tilde u}(t)}{|G|}
\leq \frac{\alpha(G)}{4}.
\end{equation*}
Hence, by Lemma \ref{propasi},
\begin{equation*}
\alpha(\tilde u>t) \geq \frac{\alpha(G)}{2}
\qquad \text{for all } t\in A.
\end{equation*}
Let $s_G := \sup A$. It remains to estimate the left-hand side of \eqref{common41e54} from below. We have
\begin{equation*}
\begin{aligned}
\frac{2}{\gamma_n}\int_0^{|G|}\int_0^{\tilde u^*(r)} \alpha^2(\tilde u>t)\,dt\,dr
&\geq \frac{2}{\gamma_n}\int_0^{|G|}
\int_0^{\min\{\tilde u^*(r), s_G\}} \alpha^2(\tilde u>t)\,dt\,dr \\
&\geq \frac{\alpha^2(G)}{2\gamma_n}
\int_0^{|G|} \min\{\tilde u^*(r), s_G\}\,dr.
\end{aligned}
\end{equation*}
Since $\mu_{\tilde u}$ is right-continuous and $\tilde u^*$ is non-increasing,
\begin{equation*}
\left|\{\tilde u^*(r) \geq \tilde u^*(s_G)\}\right|
= \lim_{s\to s_G^-} |\{\tilde u^*(r) > \tilde u^*(s)\}|
\geq |G|\left(1-\frac{\alpha(G)}{4}\right).
\end{equation*}
Therefore,
\begin{equation*}
\begin{aligned}
\frac{\alpha^2(G)}{2\gamma_n}
\int_0^{|G|} \min\{\tilde u^*(r), s_G\}\,dr
&\geq \frac{\alpha^2(G)}{2\gamma_n}
\int_0^{|G|\left(1-\frac{\alpha(G)}{4}\right)}
\min\{\tilde u^*(r), s_G\}\,dr \\
&= \frac{\alpha^2(G)}{2\gamma_n}
s_G\,|G|\left(1-\frac{\alpha(G)}{4}\right)
\geq \frac{\alpha^2(G)}{4}\,s_G\,|G|.
\end{aligned}
\end{equation*}
This concludes the proof.
\end{proof}
It remains to bound from below the quantity $s_G$ in terms of asymmetry using the approach contained in \cite{Kim}. 

We briefly comment on the additional assumption $\abs{S} \leq \frac{3}{4} \abs{G}$: it may seem unusual, but if $S$ is allowed to be closer and closer to $G$, then the torsion of $G\setminus S$ is getting closer and closer to 0. This remark suggests us that an assumption about the ratio between the measure of $G$ and that of $S$ is necessary.

\begin{prop}\label{lemma4.2}
   Let $\Omega\subset\R^2$ satisfy Definition \ref{multiplyconnected}. If $$\abs{S} \leq \frac{3}{4}\abs{G},$$ then
    \[
    T(\Omega^\OO)- T(\Omega)\geq \frac{\abs{G}^2 }{2^8 \pi \gamma_2} {\alpha(G)^3}.\]
\end{prop}

\begin{proof}
     We define $t_1$ such that 
    $$
   \mu_{\tilde V}(2t_1) = \abs{G}\left(1- \frac{1}{8} \alpha(G)\right).
    $$
    By $\abs{S} \leq \frac{3}{4}\abs{G}$, we have $\mu_{\tilde V}(2t_1) \geq \abs{S}$. Therefore, the following holds
    \begin{equation*}
        \frac{\abs{G}}{8} \alpha(G) =\abs{G}-\abs{G}\left(1- \frac{1}{8} \alpha(G)\right)= \mu_{\tilde V}(0)- \mu_{\tilde V}(2t_1)=- \int_{0}^{2t_1} \mu_{\tilde V}'(s)\,ds= \int_{0}^{2t_1} {4\pi}= 8\pi t_1.
    \end{equation*}
    Now let us distinguish 2 cases:
    \begin{itemize}
        \item  $s_G \geq t_1 \geq \dfrac{\abs{G}}{2^6\pi} \alpha(G)$, then from \eqref{essamitorsion}, we obtain
        \begin{equation}\label{caso1}
         T(\Omega^\OO)- T(\Omega)\geq  \dfrac{\abs{G}^2}{2^8\pi\gamma_2} \alpha^3(G);
        \end{equation}
        \item $s_G <t_1$, then by Corollary \ref{puntualedim2}, we have
        \begin{equation}
        \label{93}
        T(\Omega^\OO)-T(\Omega) = 
              {\lVert\tilde V\rVert_1}-{\norma{\tilde u}_1}  = \int_0^{\norma{\tilde V}_\infty} \left(\mu_{\tilde V}(t)-\mu_{\tilde u}(t)\right)\, dt \geq \int_{t_1}^{2t_1} \left(\mu_{\tilde V}(t)-\mu_{\tilde u}(t)\right)\, dt.
        \end{equation}
        Since $s_{G} < t_1 \leq t \leq 2t_1$, we have both
        \begin{equation}
        \label{nufinale}
        \mu_{\tilde V}(t) \geq \mu_{\tilde V}(2t_1)= \abs{G}\left(1- \frac{1}{8} \alpha(G)\right),
        \end{equation}
        and
        \begin{equation}\label{mufinale}
            \mu_{\tilde u}(t) \leq \abs{G}\left(1- \frac{1}{4} \alpha(G)\right).
        \end{equation}
      
        Consequently, combining \eqref{93}, \eqref{nufinale}, and \eqref{mufinale} we have
        \begin{equation}\label{caso2}
        \begin{aligned}
            {\lVert\tilde V\rVert_1}-{\norma{\tilde u}_1}&\geq \int_{t_1}^{2t_1} \left(\mu_{\tilde V}(t)-\mu_{\tilde u}(t)\right)\, dt\geq \frac{\abs{G}}{2^3} \alpha(G)\int_{t_1}^{2t_1} \, dt\\&\geq 
            \frac{\abs{G}}{2^3} \alpha(G)t_1 = 
            \frac{\abs{G}^2}{2^9\pi} \alpha^2(G) 
            \geq
            \frac{\abs{G}^2}{2^{10}\pi} \alpha^3(G).
        \end{aligned}
        \end{equation}   
    \end{itemize}
    Therefore, the thesis follows by taking the constant to be the minimum of the quantities in \eqref{caso1} and \eqref{caso2}.
\end{proof}

\section{Asymmetry of the holes}\label{section5}This section is devoted to proving that the Fraenkel asymmetry index of the holes $S$ can be controlled by the deficit in the P\'olya–Weinstein inequality. Recall also that the tilde notation is used to denote the extension of a function by constant values on the holes of the domain.

To prove this property, we replace the set $S$ with 
\[
\tilde S = D_{\tilde u}(|S|),
\] 
 since the holes are not necessarily level sets; introducing $\tilde S$ allows us to apply rearrangement techniques. Using the definition of $D_{\tilde u}(s)$ given in Section \ref{section2}, we establish the following two lemmas concerning  the asymmetry of $\tilde S$ and  the difference between $S$ and $\tilde S$. We emphasize that 
\[
\tilde S \subset G.
\]

\begin{lemma}\label{lemma5.1}
	Let $\Omega\subset\R^2$ satisfy Definition \ref{multiplyconnected}. Let $\tilde{S}=D_{\tilde u}(\abs{S})$ then 
\begin{equation*}
		T(\Omega^\OO)-T(\Omega) \geq \frac{\abs{S}}{4\pi}\abs{S\setminus\tilde{S}}
\end{equation*}
	
\end{lemma}

\begin{lemma}\label{lemma5.2}
	Let $\Omega\subset\R^2$ satisfy Definition \ref{multiplyconnected}. Let $\tilde{S}=D_{\tilde u}(\abs{S})$, if
    \[\abs{S}\leq\frac{2}{3}\abs{G}.\]
    Then 
    \begin{equation*}
    T(\Omega^\OO)-T(\Omega) \geq \frac{\abs{S}^2\alpha(\tilde{S})^3}{3^22^4\pi \gamma_ 2}.
\end{equation*}
\end{lemma}
Here, again, we remark that an assumption about the ratio between the measure of $G$ and that of $S$ is necessary for these kind of proofs.

Let us start with Lemma \ref{lemma5.1}. 

\begin{proof}[Proof of Lemma \ref{lemma5.1}]
	By definition, \begin{equation*}      T(\Omega^\OO)=\int_0^\abs{G}\tilde V^\ast(s)\,ds=\int_\abs{S}^\abs{G}V^\ast(s)\,ds+V^\ast(0)\abs{S},
	\end{equation*}
	and by remark  Corollary \ref{puntualedim2}
	\begin{equation*}
		T(\Omega)=\int_0^\abs{G}\tilde u^\ast(s)\,ds\leq \int_\abs{S}^\abs{G}\tilde u^\ast(s)\,ds+V^\ast(0)\abs{S}.
	\end{equation*}
	Hence, using the definition above and \eqref{puntuale} the term $T(\Omega^\OO)-T(\Omega)$ can be bounded from below as
	\begin{equation*}
		T(\Omega^\OO)-T(\Omega)\geq\int_\abs{S}^\abs{G}V^\ast(s)-\tilde u^\ast(s)\,ds\geq \frac{1}{4\pi}\int_\abs{S}^\abs{G}\int_s^\abs{G} \nu^\ast(r)^{-1}-\zeta^{-1}(r)\,dr\,ds,
	\end{equation*}
	and exchanging the order of integration, we have
	\begin{equation}
    \label{28}
		T(\Omega^\OO)-T(\Omega)\geq\frac{1}{4\pi}\int_\abs{S}^\abs{G}\int_0^r
		\nu^\ast(r)^{-1}-\zeta^{-1}(r)\,ds\,dr\geq\frac{\abs{S}}{4\pi}\int_\abs{S}^\abs{G}1-\zeta^{-1}(r)\,dr,
	\end{equation}
	the last inequality is due to $\nu^\ast(r)^{-1}=1$ whenever $r\geq\abs{S}$. But the definition \eqref{zeta} of $\zeta^{-1}$ characterizes the integrand term as
	\begin{equation*}%\label{1-zeta}
		1-\zeta^{-1}(r)=\begin{cases}
			\displaystyle{0} & \text{ if } r\in\left(0,\mu_u(c_m)\right)\\
			\displaystyle{1} & \text{ if } r\in\left(\mu_u(c_m),\mu_u(c_m)+|\Omega_m|\right) \text{or }r\in I_h\\
			\displaystyle{0} & \text{ if } r\in J_h
		\end{cases}.
	\end{equation*}
	Thus, the integral in  \eqref{28} is
	\begin{equation*}
		T(\Omega^\OO)-T(\Omega)\geq\frac{\abs{\tilde{S}}}{4\pi}\left(\sum_{h:\mu_u(c_h)>\abs{\tilde{S}}}\abs{\Omega_h}+\mu_u(c_k)+\abs{\Omega_k}-\abs{\tilde{S}}\right)=\frac{\abs{S}}{4\pi}\abs{\bigcup_{i=1}^m\Omega_i\setminus\tilde{S}},
	\end{equation*}
	whereas the term out of the summation appears in the case where there exists some $1\leq k\leq m$ such that $\mu_u(c_k)<\abs{\tilde{S}}<\mu_u(c_k)+\abs{\Omega_k}$ and moreover $\abs{\tilde S}=\abs{S}$ by definition. 
\end{proof}

 Now the idea is to obtain an estimate on the asymmetry of $\tilde S=D_{\tilde u}(\abs{S})$ by "propagating" the asymmetry between level set as in \cite{BD}. However, now propagation must be done outward and is possible by exploiting the following adaptation of the Lemma \ref{propasi}.
\begin{lemma}\label{propasiback}
    Let $\Omega \subset U \subset \mathbb{R}^n$ be sets of finite measure. Suppose that 
    \[ \frac{|U \setminus \Omega|}{|\Omega|} \leq \frac{1}{4} \alpha(\Omega), \] 
    then, it holds: 
    \[ \alpha(U) \geq \frac{1}{3} \alpha(\Omega). \]
\end{lemma}

\begin{proof}
    Let $B$ be a ball achieving the infimum for $\alpha(U)$, so that $|B| = |U|$ and $\alpha(U) = \frac{|U \triangle B|}{|U|}$. Using the properties of the symmetric difference and the triangle inequality ($|B \triangle C| - |B \triangle A|\leq|A \triangle C|$), we have:
    \begin{equation*}
        |U \triangle B| \geq |\Omega \triangle B| - |U \triangle \Omega|.
    \end{equation*}
    Since $\Omega \subset U$, we have $|U \triangle \Omega| = |U \setminus \Omega| = |U| - |\Omega|$. Dividing by $|U|$, we obtain:
    \begin{equation*}
        \alpha(U) \geq \frac{|\Omega|}{|U|} \left( \frac{|\Omega \triangle B|}{|\Omega|} - \frac{|U \setminus \Omega|}{|\Omega|} \right).
    \end{equation*}
    Now, let $B'$ be a ball concentric to $B$ with measure $|B'| = |\Omega|$. Since $\Omega \subset U$, $B'$ is a subset of $B$. Thus, $|B \triangle B'| = |B| - |B'| = |U| - |\Omega| = |U \setminus \Omega|$. Applying the triangle inequality again:
    \begin{equation*}
        |\Omega \triangle B| \geq |\Omega \triangle B'| - |B \triangle B'| = |\Omega \triangle B'| - |U \setminus \Omega|.
    \end{equation*}
    By the definition of $\alpha(\Omega)$, we know $\frac{|\Omega \triangle B'|}{|\Omega|} \geq \alpha(\Omega)$. Substituting this back:
    \begin{equation*}
        \alpha(U) \geq \frac{|\Omega|}{|U|} \left( \alpha(\Omega) - 2\frac{|U \setminus \Omega|}{|\Omega|} \right).
    \end{equation*}
    From the assumption $\frac{|U \setminus \Omega|}{|\Omega|} \leq \frac{1}{4}\alpha(\Omega)$ and the fact that $\alpha(\Omega) \leq 2$, we have $\frac{|U \setminus \Omega|}{|\Omega|} \leq \frac{1}{2}$. This implies:
    \begin{equation*}
        \frac{|\Omega|}{|U|} = \frac{|\Omega|}{|\Omega| + |U \setminus \Omega|} = \frac{1}{1 + \frac{|U \setminus \Omega|}{|\Omega|}} \geq \frac{1}{1 + 1/2} = \frac{2}{3}.
    \end{equation*}
    Finally, combining these inequalities:
    \begin{equation*}
        \alpha(U) \geq \frac{2}{3} \left( \alpha(\Omega) - \frac{1}{2}\alpha(\Omega) \right) = \frac{2}{3} \cdot \frac{1}{2}\alpha(\Omega) = \frac{1}{3}\alpha(\Omega).
    \end{equation*}
\end{proof}
Let us define the new threshold as
\begin{equation}\label{thresholdAs}
\begin{aligned}
        s_{\tilde{ S}}&=\inf\left\{ \tau\geq 0: \mu_{\tilde u}(\tau)\leq \abs{ S }\left(1+\dfrac{\alpha(\tilde S )}{4}\right)\right\}\\
    &=\sup\left\{ \tau \geq 0 : \mu_{\tilde u}(\tau)> \abs{ S }\left(1+\dfrac{\alpha(\tilde S )}{4}\right)\right\} \in \R.
    \end{aligned}
\end{equation}
Then, the following lemma holds.
\begin{lemma}
Let $\Omega\subset\R^n$ satisfy Definition \ref{multiplyconnected}. Then, it holds
    \begin{equation}
        \label{essastorsion}
         T(\Omega^\OO)- T(\Omega)\geq  (u^\ast\left(\abs{S}\right) -s_{\tilde{S}})\abs{S}\frac{\alpha(\tilde{ S})^2}{9\gamma_n},
    \end{equation}

    where $\gamma_n$ is the constant appearing in the quantitative isoperimetric inequality \eqref{quant_isop}.
\end{lemma}

\begin{proof}
Arguing as in Lemma \ref{4.1}, we obtain
\begin{equation}
\label{lemma5.4}
\frac{2}{\gamma_n}\int_{0}^{|G|}\int_0^{\tilde u^*(r)} \alpha^2(\tilde u>t)\,dt\,dr
\leq T(\Omega^{\mathrm{O}})-T(\Omega).
\end{equation}

We define the set
\begin{equation*}
A_{\tilde S}
:= \left\{ \tau \geq 0 : \mu_{\tilde u}(\tau)\leq |S|\left(1+\frac{\alpha(\tilde S)}{4}\right) \right\}.
\end{equation*}
Then $A_{\tilde S}$ is nonempty, since $\tilde u^*(|S|)\in A_{\tilde S}$. 
Let $s_{\tilde S}:=\inf A_{\tilde S}$ (see \eqref{thresholdAs}). 
By the monotonicity of $\mu_{\tilde u}$, the set $A_{\tilde S}$ is an interval.

Moreover, for every $\tau \in A_{\tilde S}$, we have
\begin{equation*}
\frac{|\tilde S \setminus \{\tilde u>\tau\}|}{|S|}
= \frac{\mu_{\tilde u}(\tau)}{|S|}-1
\leq \frac{\alpha(\tilde S)}{4}.
\end{equation*}
Hence, by Lemma \ref{propasiback}, it follows that
\begin{equation*}
\alpha(\tilde u>\tau)\geq \frac{\alpha(\tilde S)}{3}
\qquad \text{for all } \tau \in A_{\tilde S}.
\end{equation*}

On the other hand, using the right-continuity of the distribution function and the monotonicity of $\tilde u^*$, we have
\begin{equation*}
\left|\{\tilde u^*(r)\geq \tilde u^*(|S|)\}\right|
= \lim_{s\to |S|^-} \left|\{\tilde u^*(r) > \tilde u^*(s)\}\right|
\geq |S|.
\end{equation*}

Therefore, from \eqref{lemma5.4}, by monotonicity of the integral and the definition of $A_{\tilde S}$, we deduce
\begin{equation*}
\begin{aligned}
T(\Omega^{\mathrm{O}})-T(\Omega)
&\geq \frac{2}{\gamma_n}\int_0^{|G|}\int_0^{\tilde u^*(r)} \alpha^2(\tilde u>t)\,dt\,dr \\
&\geq \frac{2}{\gamma_n}\int_0^{|S|}\int_{s_{\tilde S}}^{\tilde u^*(r)} \alpha^2(\tilde u>t)\,dt\,dr \\
&\geq \frac{2}{9\gamma_n}\alpha^2(\tilde S)
\int_0^{|S|} \bigl(\tilde u^*(r)-s_{\tilde S}\bigr)\,dr \\
&\geq \frac{2}{9\gamma_n}\alpha^2(\tilde S)\,|S|\,
\bigl(\tilde u^*(|S|)-s_{\tilde S}\bigr).
\end{aligned}
\end{equation*}
\end{proof}
Again using the techniques contained in \cite{Kim} we are able to estimate the threshold $s_{\tilde S}$ in terms of the asymmetry, proving Lemma \ref{lemma5.2}.
\begin{proof}[Proof of Lemma \ref{lemma5.2}]
    We define $t_1$ such that
    \begin{equation*}
        \mu_{\tilde V}(\tilde V(\abs{S})-2t_1)=\abs{S}\left(1+\frac{1}{4}\alpha(\tilde{S})\right).
    \end{equation*}

        By $\abs{S} \leq \frac{2}{3}\abs{G}$, we have $ \mu_{\tilde V}(\tilde V(\abs{S})-2t_1) \leq \abs{G}$. Therefore,  the following  holds
    \begin{equation*}
        \begin{aligned}
            \frac{\abs{S}\alpha(\tilde{S})}{4}=\abs{S}\left(1+\frac{1}{4}\alpha(\tilde{S})\right)-\abs{S}=\mu_{\tilde V}(\tilde V(\abs{S})-2t_1)-\mu_{\tilde V}(\tilde V(\abs{S}))=\\-\int_{V(\abs{S})-2t_1}^{V(\abs{S})}\mu_{\tilde V}'(s)\,ds=4\pi\int_{V(\abs{S})-2t_1}^{V(\abs{S})}\,ds=8\pi t_1,
        \end{aligned}
    \end{equation*}
 where the last equality follows from symmetry. Indeed, since $\tilde V$ is radial, its superlevel sets are balls. Hence
\begin{equation*}
P(\tilde V>t)^2 = 4\pi\,\mu_{\tilde V}(t).
\end{equation*}
Moreover, for every $t \in (0, V(|S|))$ we have
\begin{equation*}
\begin{aligned}
P(\tilde V>t)^2
&=\left(-\frac{d}{dt}\int_t^\infty P(\tilde V>s)\,ds\right)^2 \\
&=\left(-\frac{d}{dt}\int_{\{\tilde V>t\}}|\nabla V|\,dx\right)^2 \\
&=\left(-\frac{d}{dt}\int_{\{\tilde V>t\}} \,dx\right)
  \left(-\frac{d}{dt}\int_{\{\tilde V>t\}} |\nabla V|^2\,dx\right) \\
&= -\mu_{\tilde V}'(t)\,\mu_{\tilde V}(t),
\end{aligned}
\end{equation*}
where in the last step we used the coarea formula.

Combining the two identities, we have
\[
-\mu_{\tilde V}'(t)\,\mu_{\tilde V}(t) = 4\pi\,\mu_{\tilde V}(t),
\]
and therefore, for every $t$ such that $\mu_{\tilde V}(t)>0$,
\begin{equation}\label{muprimo=4pi}
-\mu_{\tilde V}'(t) = 4\pi.
\end{equation} 
     Also the pointwise comparison gave us a bound on $s_{\tilde{S}}$ as
    \begin{equation}\label{boundss}
        \tilde V(\abs{S})-2t_1=\tilde V\left(\abs{S}\left(1+\frac{1}{4}\alpha(\tilde{S})\right)\right)\geq\tilde u^\ast\left(\abs{S}\left(1+\frac{1}{4}\alpha(\tilde{S})\right)\right)=s_{\tilde{S}}.
    \end{equation}
    Now let us distinguish two cases:
    \begin{itemize}
        \item if $\tilde V(\abs{S})-\tilde u^\ast(\abs{S})<t_1$, then from \eqref{essastorsion} and \eqref{boundss}\, we obtain
        \begin{equation*}
            {\tilde u}^\ast(\abs{S})-s_{\tilde{S}}\geq\tilde u^\ast(\abs{S})-\tilde V(\abs{S})+2t_1\geq -t_1+2t_1=\frac{\abs{S}\alpha(\tilde{S})}{2^5\pi},
         \end{equation*}
        hence, we have

         \begin{equation*}
              T(\Omega^\OO)-T(\Omega)\geq   (u^\ast\left(\abs{S}\right) -s_{\tilde{S}})\abs{S}\frac{2\alpha({\tilde{S}})^2}{9\gamma_n} \geq \frac{\abs{S}^2\alpha(\tilde{S})^3}{3^22^4\pi \gamma_ 2}.
         \end{equation*}

         \item if $\tilde V(\abs{S})-\tilde u^\ast(\abs{S})\geq t_1$, we define the function \begin{equation*}%\label{funzionefbound}
          f(s)=
         \begin{cases}
         \tilde V (\abs{S}) & \text{ if } s \leq \abs{S}\\
             u^\ast (\abs{S})& \text{ if }  \abs{S} < s \leq \abs{S} + 4\pi \left(\tilde V(\abs{S})-\tilde u^\ast(\abs{S})\right) \\
              \tilde V(s)& \text{ if }  \abs{S} + 4 \pi \left(\tilde V(\abs{S})-\tilde u^\ast(\abs{S})\right)   <s \leq \abs{G},
         \end{cases}\end{equation*}
         
         which is greater or equal to $\tilde u^\ast$ (see figure \ref{fig:funzionef}).
         
         Hence, we have 
         \begin{equation}
             \label{claimbho}
             \begin{aligned}
                 T(\Omega^\OO)-T(\Omega)&=\int _0^{\abs{G}} \left(\tilde V(s)-\tilde u^\ast(s)\right)\, ds \\
                 &\geq \int _0^{\abs{G}} \left(\tilde V(s)-f(s)\right)\, ds= 2 \pi \left(\tilde V(\abs{S})-\tilde u^\ast(\abs{S})\right)^2.
             \end{aligned}
             \end{equation} 

             The latter, after some computations,  gives
              \begin{equation*}
             T(\Omega^\OO)-T(\Omega)\geq 2 \pi \left(\tilde V(\abs{S})-\tilde u^\ast(\abs{S})\right)^2\geq 2 \pi t_1^2\geq \frac{2 \pi}{2^{10} \pi^2}\abs{S}^2\alpha(\tilde{S})^2\geq \frac{\abs{S}^2\alpha(\tilde{S})^3}{2^{10} \pi}.
         \end{equation*}

    \end{itemize}
      \begin{figure}[ht!]
        \centering

\tikzset{every picture/.style={line width=0.75pt}} %set default line width to 0.75pt        

\begin{tikzpicture}[x=0.75pt,y=0.75pt,yscale=-1,xscale=1]

%Straight Lines [id:da5594649633021156] 
\draw  [dash pattern={on 0.84pt off 2.51pt}]  (292.5,104.7) -- (292.55,266.32) ;
%Straight Lines [id:da5942426507868859] 
\draw [color={rgb, 255:red, 0; green, 0; blue, 0 }  ,draw opacity=1 ] [dash pattern={on 0.84pt off 2.51pt}]  (179.5,103.78) -- (188.75,103.82) -- (263,104.09) ;
%Shape: Right Triangle [id:dp26763058469178014] 
\draw  [draw opacity=0][fill={rgb, 255:red, 65; green, 117; blue, 5 }  ,fill opacity=0.58 ] (292.5,79.8) -- (314,104.7) -- (292.5,104.7) -- cycle ;
%Shape: Axis 2D [id:dp7032668208380202] 
\draw  (145,266.4) -- (490.5,266.4)(179.55,15.3) -- (179.55,294.3) (483.5,261.4) -- (490.5,266.4) -- (483.5,271.4) (174.55,22.3) -- (179.55,15.3) -- (184.55,22.3)  ;
%Straight Lines [id:da2947591129314542] 
\draw [color={rgb, 255:red, 65; green, 117; blue, 5 }  ,draw opacity=1 ]   (179,79.5) -- (292.5,79.8) ;
%Straight Lines [id:da15806518879650078] 
\draw [color={rgb, 255:red, 65; green, 117; blue, 5 }  ,draw opacity=1 ]   (292.5,79.8) -- (448.5,266.5) ;
%Straight Lines [id:da5484709852409027] 
\draw [color={rgb, 255:red, 0; green, 0; blue, 0 }  ,draw opacity=1 ]   (179,84) -- (247.5,84.3) ;
%Straight Lines [id:da5662895113598212] 
\draw [color={rgb, 255:red, 0; green, 0; blue, 0 }  ,draw opacity=1 ]   (247.5,84.3) -- (263,104.09) ;
%Straight Lines [id:da5011809630795425] 
\draw    (292.5,104.39) -- (311,126.84) ;
%Straight Lines [id:da8674515481041667] 
\draw [color={rgb, 255:red, 0; green, 0; blue, 0 }  ,draw opacity=1 ]   (311,126.84) -- (329.5,149.28) ;
%Straight Lines [id:da8089452956814356] 
\draw    (430,244.05) -- (448.5,266.5) ;
%Straight Lines [id:da6572568937742018] 
\draw [color={rgb, 255:red, 0; green, 0; blue, 0 }  ,draw opacity=1 ]   (341,149.59) -- (419.5,243.75) ;
%Straight Lines [id:da8066824427181756] 
\draw [color={rgb, 255:red, 0; green, 0; blue, 0 }  ,draw opacity=1 ]   (419.5,243.75) -- (430,244.05) ;
%Straight Lines [id:da6010325590490861] 
\draw [color={rgb, 255:red, 0; green, 0; blue, 0 }  ,draw opacity=1 ]   (329.5,149.28) -- (341,149.59) ;
%Straight Lines [id:da9339825758405371] 
\draw [color={rgb, 255:red, 0; green, 0; blue, 0 }  ,draw opacity=1 ]   (263,104.09) -- (292.5,104.39) ;
%Straight Lines [id:da6689242214556868] 
\draw [dashed,color={rgb, 255:red, 208; green, 2; blue, 27 }  ,draw opacity=1 ]   (179,79.5) -- (292.5,79.8) ;
%Straight Lines [id:da9321419225717913] 
\draw [dashed,color={rgb, 255:red, 208; green, 2; blue, 27 }  ,draw opacity=1 ]   (292.5,104.39) -- (314,104.7) ;
%Straight Lines [id:da8425927555383188] 
\draw [dashed,color={rgb, 255:red, 208; green, 2; blue, 27 }  ,draw opacity=1 ]   (314,104.7) -- (448.5,266.5) ;
%Straight Lines [id:da5389907024655973] 
\draw [dashed,color={rgb, 255:red, 208; green, 2; blue, 27 }  ,draw opacity=1 ]   (292.5,79.8) -- (292.5,104.39) ;
%Straight Lines [id:da17883543304905292] 
\draw  [dash pattern={on 0.84pt off 2.51pt}]  (314,104.7) -- (314.05,266.32) ;

% Text Node
\draw (440.5,275.15) node [anchor=north west][inner sep=0.75pt]    {$\abs{G}$};
% Text Node
\draw (114.75,66.4) node [anchor=north west][inner sep=0.75pt]    {$\tilde V(\abs{S})$};
% Text Node
\draw (310,76.4) node [anchor=north west][inner sep=0.75pt]    {$\tilde V(s)$};
% Text Node
\draw (117,96.28) node [anchor=north west][inner sep=0.75pt]    {$\tilde u^\ast(\abs{S})$};
% Text Node
\draw (277.4,271.4) node [anchor=north west][inner sep=0.75pt]    {$\abs{S}$};
% Text Node
\draw (314.8,191.2) node [anchor=north west][inner sep=0.75pt]    {$\tilde u^\ast(s)$};
% Text Node
\draw (384.4,163.6) node [anchor=north west][inner sep=0.75pt]    {$f(s)$};
% Text Node
\draw (304.54,272.26) node [anchor=north west][inner sep=0.75pt]    {\scriptsize $\abs{S}+ 4 \pi (\tilde V(\abs{S})- \tilde u^\ast (\abs{S}))$};
\end{tikzpicture}
\caption{Upper bound for $\tilde u^\ast(s)$}
        \label{fig:funzionef}
    \end{figure}
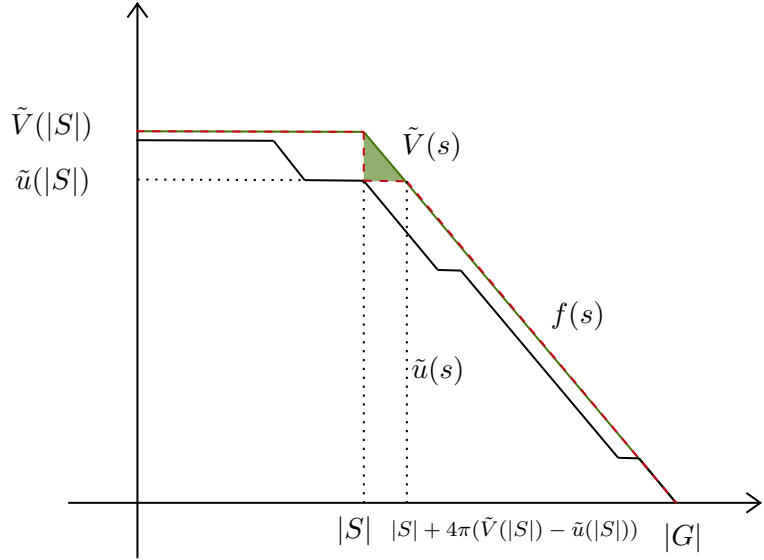
\end{proof}

By combining Lemmas \ref{lemma5.1} and \ref{lemma5.2}, we obtain the desired asymmetry control for holes.

\begin{prop}\label{prop5.5}
	Let $\Omega\subset\R^2$ satisfy Definition \ref{multiplyconnected}. Then 
	\begin{equation*}
	T(\Omega^\OO)-T(\Omega) \geq \frac{\abs{S}^2\alpha^3(S)}{3^22^{10}\pi \gamma_ 2}.
	\end{equation*}
\end{prop}

\begin{proof}
	Let $B$ be the ball which realizes $\alpha(\tilde{ S})$. It is sufficient to observe that 
	$$
	\abs{S}  \alpha(S)\leq \abs{S\triangle B}\leq \abs{\tilde{ S}\triangle B} + \abs{S\triangle \tilde{ S}}	= \abs{\tilde{ S}\triangle B} + 2 \abs{S\setminus \tilde{ S}} \leq 2\left(\abs{\tilde{ S}} \alpha(\tilde{ S})+ \abs{S\setminus \tilde{ S}}\right).$$

    Indeed, by Lemmas \ref{lemma5.1} and \ref{lemma5.2}, and the last property, we have
    \begin{equation*}
        \begin{aligned}
           	T(\Omega^\OO)-T(\Omega) &\geq \frac{\abs{S}}{4\pi}\abs{S\setminus\tilde{S}} +\frac{\abs{S}^2\alpha(\tilde{S})^3}{3^22^4\pi \gamma_ 2} \\
        &\geq \frac{1}{3^22^3\pi \gamma_ 2 \abs{S}}  \left[\abs{S\setminus\tilde{S}}^3 +\abs{S}^3\alpha(\tilde{S})^3\right]\\
         &\geq \frac{1}{3^22^6\pi \gamma_ 2 \abs{S}}  \left[\abs{S\setminus\tilde{S}}+\abs{S}\alpha(\tilde{S})\right]^3\\
         &\geq \frac{1}{3^22^{9}\pi \gamma_ 2 \abs{S}}  \left[\abs{S}\alpha(S)\right]^3.
        \end{aligned}
    \end{equation*}
	\end{proof}

Theorem \ref{thm:dueasimmetrie} follows by combining Propositions \ref{lemma4.2} and \ref{prop5.5}.
\section{Almost radiality of the solution and conclusion}\label{section6}

This section will focus on the almost radiality of the solution.
The idea is similar to the one exploited in \cite{ABCMP} and \cite{ABMP}, with some appropriate adaptations.

As a first step, we need to prove that the  following function defined on $G$ 
\begin{equation}
\label{doppiav}
    w(x):= \min\{ \tilde u(x), \, \tilde u^\ast(\abs{S})\}
\end{equation}
has the Polya-Szego deficit controlled via the torsion deficit $T(\Omega^\OO) -T(\Omega)$. Also here, we recall that the tilde notation is used to denote the extension of a function by constant values on the holes of the domain. We observe that for $w(x)$ the following properties hold:
\begin{enumerate}
    \item $\Big\lvert \{w\geq \tilde u^\ast(\abs{S})\}\Big\rvert\geq \abs{S}$;
    \item $\displaystyle{\int_G \abs{\nabla w}^2 \, dx \leq  \int_G \abs{\nabla u}^2 \, dx }$;
    \item $\displaystyle{w(x)\leq \tilde u(x) \leq w(x) + \chi_{D_{\tilde u}(\abs{S})}\left(\tilde V(\abs{S})-\tilde u^\ast(\abs{S})\right)}$.
\end{enumerate}
In particular, the upper bound in the third property follows from the fact that $w$ and $\tilde u$ differ only on $D_{\tilde u}(|S|)$. Moreover, by Corollary~\ref{puntualedim2}, we have $\tilde u^\ast(|S|) \leq \tilde V^\ast(|S|)$. In view of \eqref{puntuale}, the additional term is nonnegative and vanishes if and only if the minimum in \eqref{doppiav} is attained at \(u(x)\); in this case, equality holds. We, lastly, stress that 
\begin{equation}
\label{distrw}
     \mu_{w}(s) = \mu_{\tilde u}(s) \quad\text{ for } \quad s < \tilde u^\ast (\abs{S}).
\end{equation}
Hence we have

\begin{lemma}
\label{lem_grad}
     Let $\Omega\subset\R^2$ satisfy Definition \ref{multiplyconnected}. Let $u$ and $V$ be the torsion function of $\Omega$ and $\Omega^\OO$ respectively. Let $w$be the function defined in \eqref{doppiav}.
     
     Then 
\begin{equation}
\label{tops}
\int_{\Omega}\abs{\nabla w}^2-\int_{\Omega^\sharp}|\nabla w^\sharp|^2\leq 2\left(T(\Omega^\OO) - T(\Omega)\right).
\end{equation}
\end{lemma}
\begin{proof}
Notice that the extended torsion function $\tilde V$  is the unique minimizer of
\begin{equation*}
    -T(\Omega^\OO)= \min\left\{\int_{G^\sharp}\abs{\nabla\varphi}^2\,dx - 2\int_{G^\sharp} \varphi\, dx  : \, \varphi \in H^1_0(G^\sharp), \varphi\equiv  \, \text{constant on }  S^\sharp \right\}.
\end{equation*}
Hence we have  by P\'olya-Szeg\"o inequality and equimisurability
\begin{equation*}
%\label{usharV}
\int_G\,|\nabla u|^2-2\int_{G} u\geq \int_{G^{\sharp}}\,|\nabla u^\sharp|^2-2\int_{G^\sharp} u^\sharp\geq \int_{G^{\sharp}}\abs{\nabla \tilde V}^2-2\int_{G^\sharp}\tilde V.
\end{equation*}
 Therefore, combining the previous steps with the properties of $w$ recalled at the beginning of this section, we infer
\begin{equation*}
    \begin{aligned}
\int_{G}\abs{\nabla w}^2-\int_{G^\sharp}|\nabla w^\sharp|^2 
&\leq\int_{G}\abs{\nabla w}^2-\int_{G^\sharp}|\nabla w^\sharp|^2+\int_{\tilde u>\tilde u^\ast(\abs{S})}\abs{\nabla u}^2\,dx-\int_{\tilde u^\sharp>\tilde u^\ast(\abs{S})}\abs{\nabla u^\sharp}^2\,dx
\\
&\leq \int_{G}\abs{\nabla\tilde u}^2-\int_{G^\sharp}|\nabla\tilde u^\sharp|^2 \\ 
&\leq \int_{G^\sharp}\abs{\nabla\tilde V}^2-\int_{G^\sharp}|\nabla\tilde u^\sharp|^2\\
&\leq 2\int_{G^\sharp}(\tilde V-u^\sharp) \, dx\\
&= 2 \left( T(\Omega^\OO) - T(\Omega)\right) .
\end{aligned}
\end{equation*}
Here, the first inequality is justified since, by the P\'olya--Szeg\"o inequality applied on $\{\tilde u>\tilde u^\ast(\abs{S})\}$, the additional term is nonnegative.
\end{proof}

In particular, \eqref{tops} suggests that the P\'olya-Szeg\H o inequality for $w$ holds true almost as an equality when $T(\Omega^\OO) - T(\Omega)$ is small enough. So it is natural to consider the P\'olya-Szeg\H o quantitative inequality, recalled in Theorem \ref{polya_quant}.

To show that $w$ and $w^\sharp$ are close, it suffices to prove that, for every $\delta>0$, the quantity $M_{w^\sharp}(\delta)$, defined in \eqref{eumu}, can be bounded above by a power of $\delta$. This estimate does not hold in general for arbitrary Sobolev functions; however, it is valid for $w$ because $w$ is a truncation of the torsion function of $\Omega$.

The main idea to bound from above the quantity $M_{w^{\sharp}}(\delta)$ is to write the set $A:=\{\nabla w^\sharp < \delta\} \cap \{0 <w^\sharp < {\tilde u}^\ast(\abs{S})\}$ as the union of suitable subsets, and to bound from above the measure of each of them. In order to do that, in Lemma \ref{lemmaI}
 we define a set $I\subseteq [0,{\tilde u}^\ast(\abs{S})]$ and, in \eqref{t_epsilon}, we define a positive number $t_{\varepsilon}\in{[0, {\tilde u}^\ast(\abs{S})]}$, so that we have

\begin{align*}
   A=&\Big( A\cap w^{\sharp-1}(I^c\cap (0,t_{\varepsilon})\Big)\\
    &\cup\Big(A\cap w^{\sharp-1}(I\cap (0,t_{\varepsilon})\Big)\\
    &\cup\Big(A\cap w^{\sharp-1}(t_{\varepsilon}, +\infty)\Big).
\end{align*}
We will prove next lemmas and propositions under the additional assumptions
$$ \abs{G}=1, \quad T(\Omega^\OO) - T(\Omega)=:\varepsilon \le \varepsilon_0, $$ where $\varepsilon_0$ will be suitably chosen later.

\begin{lemma} \label{lemmaI}
Let $\varepsilon:=T(\Omega^\OO) - T(\Omega)$ and let us define 
the set $I$ as follows
   \begin{equation}
\label{I}
I=\left\{t \in [0, \tilde u^\ast (\abs{S})] :\,\int_{\tilde V = t}\abs{\nabla \tilde V}-\int_{w^\sharp=t}\abs{\nabla w^\sharp}>\varepsilon^{\frac{1}{4}}\right\}.
\end{equation}
    Then,    
    \begin{equation}\label{I_bound}
        \abs{I}\le \frac{6\abs{G}}{\sqrt{2\pi}}  \varepsilon^{\frac{1}{4}},
    \end{equation}
and, for every  $t\in I^c$,  it holds
\begin{equation}
\label{gradient}
\int_{\tilde V = t}\abs{\nabla \tilde V}\;-\int_{w^\sharp=t}|\nabla w^\sharp|\leq \varepsilon^{\frac14}.
\end{equation}
    
\end{lemma}
\begin{proof}
Claim \eqref{gradient}  is a direct consequence of the definition of the set $I$ in \eqref{I}. Moreover, using  the Coarea formula \eqref{coarea}, we have
\begin{align*}
    \varepsilon^{\frac14}\abs{I}\leq \int_{I}\biggl(\int_{\tilde V = t}\abs{\nabla \tilde V}-\int_{\{w^\sharp=t\}}|\nabla w^\sharp|\biggr)&\leq \int_0^{+\infty}\biggl(\int_{\tilde V = t}\abs{\nabla \tilde V}-\int_{w^\sharp=t}|\nabla w^\sharp|\biggr)\\&=\int_{G^\sharp}\abs{\nabla \tilde V}^2-\int_{G^\sharp}|\nabla w^\sharp|^2<\frac{6\abs{G}}{\sqrt{2\pi}} \varepsilon^\frac{1}{2}, 
    \end{align*}
and claim \eqref{I_bound} follows.
\end{proof}
\begin{lemma}
   Let $\delta>0$ and let  $\tilde V$ be the extension of the solution to \eqref{probsimm}. Then, there exists a positive constant $K$ such that
    \begin{equation*}
        \abs{\{\abs{\nabla \tilde V}\leq\delta\}\cap (G^\sharp\setminus S^\sharp)}\leq K \delta^2.
    \end{equation*}
\end{lemma}
\begin{proof}
    By direct computation, from \eqref{probsimm}, we have that $\abs{\nabla \tilde V}(x)= \abs{x}$ if $\pi\abs{x}^2 \geq \abs{S}$, otherwise it is zero.  Hence,
    \begin{itemize}
        \item if $\pi\delta^2 < \abs{S}$ we have 
          \begin{equation*}
        \{\abs{\nabla \tilde V}\leq\delta\}=S^\sharp;
    \end{equation*}
      \item if $\pi\delta^2 \geq \abs{S}$ we have 
      \begin{equation*}
        \{\abs{\nabla \tilde V}\leq\delta\} = \{\abs{x}\leq\delta\}.
    \end{equation*}

    \end{itemize}
    In any case, 
      $$\abs{\{\abs{\nabla \tilde V}\leq\delta\}\cap G^\sharp\setminus S^\sharp}\leq K \delta^2 - \abs{S}\leq K \delta^2 .$$
\end{proof}
\begin{lemma} Let $w$ be the function defined in \eqref{doppiav}.
 If we define the following quantity
\begin{equation}
    \label{t_epsilon}
    t_{\varepsilon}=\sup\left\{t>0\;:\;\;\mu_w(t)>\varepsilon^{\frac{1}{4}}\right\},
\end{equation}
then, for every $\delta>0$, it holds     \begin{equation*}
        \abs{\{ A \cap w^{\sharp-1}(I^c\cap (0,t_{\varepsilon}) }\leq\pi\left(\delta+\frac{\varepsilon^{\frac18}}{2\sqrt{\pi}}\right)^2.
    \end{equation*}
\end{lemma}
\begin{proof}
    By direct computation, property \eqref{gradient} in Lemma \ref{lemmaI} is equivalent to
    \begin{equation}\label{Icomplementare}
        P(\tilde V = t)\abs{\nabla \tilde V}(y)-P(w^\sharp=t)\abs{\nabla w^\sharp}(x)\leq\varepsilon^\frac14,
    \end{equation}
    for every $y\in\{\tilde V = t\}$ and $x\in\{w^\sharp=t\}$. If we fix $x\in\{w^\sharp=t\}$ and we consider $y=\ell(t)x,$ where $\displaystyle{\ell(t)=\left(\frac{\mu_{\tilde V}(t)}{\mu_{\tilde u}(t)}\right)^\frac{1}{2}}$, we have that $y\in\{\tilde V=t\}$. Thus by pointwise comparison, the monotonicity of the distribution function and the properties of $w$ we have that
    \begin{equation*}
        P(\tilde V = t)\geq P(w^\sharp=t)=2(\pi\mu_w(t))^{\frac 1 2}\geq2\sqrt{\pi}\varepsilon^{\frac18 },
    \end{equation*}
    which implies in \eqref{Icomplementare} that
    \begin{equation}\label{diffgrad}
        \abs{\nabla \tilde V}(\ell(t)x)-\abs{\nabla w^\sharp}(x)\leq\frac{\varepsilon^\frac14}{2\sqrt{\pi}\varepsilon^{\frac18}}=\frac{\varepsilon^{\frac18}}{2\sqrt{\pi}} .
    \end{equation}
    Moreover, combining \eqref{diffgrad}, the Definition \eqref{defV} and the fact that $\ell(t)\geq1$ we have
    \begin{equation*}
        \abs{x}-\abs{\nabla w^\sharp}(x)\leq\abs{\nabla \tilde V}(\ell(t)x)-\abs{\nabla w^\sharp}(x)\leq\frac{\varepsilon^{\frac18}}{2\sqrt{\pi}}, \quad\forall x\in w^{\sharp-1}(I^c \cap (0, t_{\varepsilon})).
    \end{equation*}
    From the above inequality, we deduce that
    \begin{equation*}
        A\cap w^{\sharp-1}(I^c\cap (0, t_{\varepsilon}))\subseteq\left\{\abs{x}\leq\delta+\frac{\varepsilon^{\frac18}}{2\sqrt{\pi}}\right\},
    \end{equation*}
hence the thesis follows.
\end{proof}

\begin{oss}
We observe that 
\begin{equation*}
\abs{w^{\sharp-1}(t_{\varepsilon}, +\infty)}\le\varepsilon^{\frac14},
\end{equation*}
indeed
    $$\mu_w(t_{\varepsilon,})=\lim_{t\rightarrow t_{\varepsilon}^+}\mu_w(t)\leq \varepsilon^{\frac14},$$
    since $\mu_w$ is right continuous.
\end{oss}

\begin{lemma}%\label{wsharpintersecatgreenllo}
    Let $w$ be the function defined in \eqref{doppiav}.Then,  we have 

\begin{equation*}
     \abs{w^{\sharp-1}(I) \cap \{0 <w^\sharp < {\tilde u}^\ast(\abs{S})\}} \leq  K \varepsilon^{\frac16}. 
\end{equation*}
\end{lemma}
\begin{proof}
    Arguing as in the proof of Theorem \ref{buonocore}, by the isoperimetric inequality and Fleming-Rishel formula we get
    \begin{equation*}
        \frac{1}{4\pi}(-\mu_{\tilde V}'(t)) =\frac{1}{4\pi}(-\mu_{\tilde V}'(t))\nu^{\ast-1}(\mu_{\tilde V}(t))\leq\frac{1}{4\pi}(-\mu_{\tilde u}'(t))\zeta^{-1}(\mu_{\tilde u}(t)) \leq \frac{1}{4\pi}(-\mu_{\tilde u}'(t)), \qquad \forall t \in I^c.
    \end{equation*}
    Note that the first equality holds because $\nu^{*-1}$ vanishes whenever the superlevel set intersects a hole. However, this cannot occur for all $t \in I^c$, because the gradient of $u$ is close to that of $V$, and since $V$ has a nonzero gradient, $u$ cannot have a vanishing gradient either; hence the superlevel sets do not lie entirely inside any hole. Therefore, we obtain
    \begin{equation*}
        -\mu_{\tilde u}'(t)\geq-\mu_{\tilde V}'(t)\quad\forall t\in I^c,
    \end{equation*}
    which gives us
    \begin{equation}\label{immagineic}
    \begin{aligned}
        \abs{w^{\sharp-1}(I^c) \cap \{0 <w^\sharp < {\tilde u}^\ast(\abs{S})\}}&=\abs{w^{\sharp-1}(I^c) \cap \left\{ \abs{ \nabla w^\sharp}=0 \right\}} + \int_{I^c}-\mu_{\tilde u}'(t)\,dt \\ &\geq \int_{I^c}-\mu_{\tilde u}'(t)\,dt\\&\geq\int_{I^c}-\mu_{\tilde V}'(t)\,dt=\abs{\tilde V^{-1}(I^c) \cap (G^\sharp\setminus S^\sharp)}.
    \end{aligned}
    \end{equation}    So,  we remark that by definition $w^{\sharp-1}([0, \tilde u ^\ast (\abs{S})])=G^\sharp\cap \{0 <w^\sharp < {\tilde u}^\ast(\abs{S})\}$ where $I$ is defined in \eqref{I} and using \eqref{immagineic} we have that
    \begin{equation*}
    \begin{aligned}
         \abs{w^{\sharp-1}(I) \cap \{0 <w^\sharp < {\tilde u}^\ast(\abs{S})\}}&=\abs{G^\sharp \cap  \{0 <w^\sharp < {\tilde u}^\ast(\abs{S})\}}-\abs{w^{\sharp-1}(I^c) \cap \{0 <w^\sharp < {\tilde u}^\ast(\abs{S})\}}\\ &\leq \abs{G^\sharp \setminus  S^\sharp}-\abs{\tilde V^{-1}(I^c)\cap G^\sharp\setminus S^\sharp}=\abs{\tilde V^{-1}(I)\cap G^\sharp\setminus S^\sharp},
    \end{aligned}
    \end{equation*}
    but the last term can be bounded as follows
    \begin{equation*}
        \begin{aligned}
           \abs{\tilde V^{-1}(I)\cap G^\sharp\setminus S^\sharp}&=\int_I\int_{\tilde V = t}\frac{1}{\abs{\nabla \tilde V}}\,d\mathcal{H}^{n-1}\,dt\\&=\int_I\int_{\{\tilde V = t\}\cap\{\abs{\nabla \tilde V}\leq\delta\}}\frac{1}{\abs{\nabla \tilde V}}\,d\mathcal{H}^{n-1}\,dt+\int_I\int_{\{\tilde V = t\}\cap\{\abs{\nabla \tilde V}>\delta\}}\frac{1}{\abs{\nabla \tilde V}}\,d\mathcal{H}^{n-1}\,dt\\&\leq\abs{\{\abs{\nabla \tilde V}\leq\delta\} \cap (G^\sharp \setminus S^\sharp)}+\frac{1}{\delta}\int_I2\pi\left(\frac{\mu_{\tilde V}(t)}{\pi}\right)^\frac{1}{2}\,dt\\&\leq\abs{\{\abs{\nabla \tilde V}\leq\delta\}\cap (G^\sharp \setminus S^\sharp)}+\frac{2\sqrt{\pi}}{\delta}\abs{I}\\&\leq K \delta^2+\frac{2\sqrt{\pi}}{\delta}\varepsilon^{\frac 14}.
        \end{aligned}
    \end{equation*}

    The thesis follows by choosing $\delta=\varepsilon^\frac{1}{12}$.
\end{proof}

 \begin{prop}\label{Musharp}
 Let $w$ be the function defined in \eqref{doppiav} and let $M_{w^\sharp}$ be the quantity defined in \eqref{eumu}. Ler $r>0$ the positive constant given in Theorem \ref{polya_quant}. Then,  it holds
    
   $$ M_{w^\sharp}(\varepsilon^r)\leq K\varepsilon^{\theta},$$
   where $K$ is a positive constant and $\theta=\theta(r)>0$.
\end{prop}
\begin{proof}
We have 
\begin{align*}
   M_{w^\sharp}(\delta) =& \abs{\{\nabla w^\sharp < \delta\} \cap \{0 <w^\sharp < {\tilde u}^\ast(\abs{S})\}} = \abs{A} \\=&\abs{A\cap w^{\sharp-1}(I^c\cap (0,t_{\varepsilon})}\\
    &+\abs{A\cap w^{\sharp-1}(I\cap (0,t_{\varepsilon})}\\
    &+\abs{A\cap w^{\sharp-1}(t_{\varepsilon}, +\infty)}\\
    \leq &\abs{A\cap w^{\sharp-1}(I^c\cap (0,t_{\varepsilon})}
\\
&+\abs{w^{\sharp-1}(I) \cap \{0 <w^\sharp < {\tilde u}^\ast(\abs{S})\}}  
 +\mu_w(t_{\varepsilon}) 
\\ \leq &\pi\left(\delta+\frac{\varepsilon^{\frac18}}{2\sqrt{\pi}}\right)^2 + K \varepsilon^{\frac16}+2\sqrt{\pi}\varepsilon^{\frac16}+ \varepsilon^{\frac14}.
\end{align*}
  So, evaluating $M_{u^\sharp}(\delta)$ at $\delta=\varepsilon^r$, we obtain
  $$M_{w^\sharp}(\varepsilon^r)\leq \pi\left(\varepsilon^r+\frac{\varepsilon^{\frac18}}{2\sqrt{\pi}}\right)^2 + K \varepsilon^{\frac16}+{2\sqrt{\pi}}\varepsilon^{\frac 16} + \varepsilon^{\frac14},$$
 so we can conclude, by setting

 $$\theta=\min\left\{ 2r, \frac 16 \right\},$$
obtaining the desired claim for a positive constant $K$.
\end{proof}
Hence, combining the previous results we have
\begin{theorem}
\label{anellia}
    Let $u$ be the maximizer of \eqref{torsione} and $\tilde u$ its constant extension. Then, there exist two positive constants $\tilde \theta$  and $\tilde C(\abs{G},\abs{S})$ such that
    \[
    \inf_{x_0\in\R^2}\norma{\tilde u-\tilde u^\sharp(\cdot+x_0)}_{L^1(\R^2)}\leq\tilde C\varepsilon^{\tilde \theta}.
    \]
\end{theorem}
\begin{proof}
  By the definition of  $w$ in \eqref{doppiav}, we have that 
    \begin{equation*}
        w(x)\leq u(x)\leq w(x)+\chi_{\tilde S}(\tilde V(\abs{S})-\tilde u^*(\abs{S})),
    \end{equation*}
    and similarly
    \begin{equation*}
        w^\sharp(x)\leq u^\sharp(x)\leq w^\sharp(x)+\chi_{ S^\sharp}(\tilde V(\abs{S})-\tilde u^*(\abs{S})),
    \end{equation*}
    where $\tilde S=D_{\tilde u}(\abs{S})$ and $\tilde V$ is the extension of the solution to \eqref{probsimm}. Hence
    \begin{equation*}
       w(x)- w^\sharp(x)-\chi_{ S^\sharp}(\tilde V(\abs{S})-\tilde u^*(\abs{S}))\leq u(x)-u^\sharp(x)\leq     w(x)- w^\sharp(x)+\chi_{ \tilde S}(\tilde V(\abs{S})-\tilde u^*(\abs{S})),
    \end{equation*}
    and, by integrating, we have
    \begin{equation}\label{umenousharp}
        \int_{\R^2}\abs{\tilde u-\tilde u^\sharp}\,dx\leq \int_{\R^2}\abs{w-w^\sharp}\,dx+\abs{S}(\tilde V(\abs{S})-\tilde u^*(\abs{S})).
    \end{equation}
    On the other hand, by Theorem \ref{polya_quant} we have that there exist two positive constants $r,s$ such that
    \begin{equation}\label{wmenowsharp}
        \inf_{x_0\in\R^2}\norma{w-w^\sharp(\cdot+x_0)}_{L^1(\R^2)}\leq C\left[M_{w^\sharp}(E(w)^r)+E(w)\right]^s,
    \end{equation}
    but now we are in position to apply estimates of lemmas \ref{lem_grad} and \ref{Musharp}. Therefore, by combining \eqref{umenousharp} and \eqref{wmenowsharp}
    \begin{equation*}
        \inf_{x_0\in\R^2}\norma{\tilde u-\tilde u^\sharp(\cdot+x_0)}_{L^1(\R^2)}\leq\overline C\varepsilon^{s\theta}+\frac{1}{2\pi}\varepsilon^\frac{1}{2},
    \end{equation*}
    where the bound of the last term follows from \eqref{claimbho} and we can conclude, by setting
    \begin{equation*}
        \tilde \theta=\min\left\{s\theta,\frac{1}{2}\right\},\qquad \tilde C=\max\left\{\overline C,\frac{1}{2\pi}\right\}.
    \end{equation*}
    Note that $\tilde \theta=\tilde \theta(s,r)$ depends on $r,s$ the constant given in Theorem \ref{polya_quant}.
    
    Finally, we remove the additional assumption $\abs{G}=1$ using the rescaling properties of the torsional rigidity.
\end{proof}

By the previous theorem, the torsion function $u$ and its truncation $w$ are almost symmetric; consequently, each superlevel set is close to a ball. We now prove Theorem \ref{thm:anellia}.

\begin{proof}[Proof of Theorem \ref{thm:anellia}]
Starting from Theorem \ref{anellia}, and assuming that the infimum is attained at $x_0=0$, we have
\begin{equation*}
C\varepsilon^{\tilde\theta}
\geq \|\tilde u-\tilde u^\sharp\|_{L^1(\mathbb{R}^n)}
\geq \|w-w^\sharp\|_{L^1(\mathbb{R}^n)}
= \int_0^{\tilde u^*(|S|)} |\{w>t\}\triangle\{w^\sharp>t\}|\,dt.
\end{equation*}
Define
\begin{equation*}
L=\left\{t\in(0,\tilde u^*(|S|)):\,
|\{w>t\}\triangle\{w^\sharp>t\}|\geq \varepsilon^{\frac{\tilde\theta}{2}}\right\}.
\end{equation*}
Arguing as in Lemma \ref{lemmaI}, we obtain $|L|\leq C\varepsilon^{\frac{\tilde\theta}{2}}$ for some constant $C>0$. In particular,
\begin{equation}
\label{measureofT}
|\{w>t\}\triangle\{w^\sharp>t\}|
< \varepsilon^{\frac{\tilde\theta}{2}}
\qquad \forall t\in(0,\tilde u^*(|S|))\setminus L.
\end{equation}
Hence, there exists $\overline t \leq 2C\varepsilon^{\frac{\tilde\theta}{2}}$ such that
\begin{equation}
\label{epsgamm}
|\{w>\overline t\}\triangle\{w^\sharp>\overline t\}|
< \varepsilon^{\frac{\tilde\theta}{2}}.
\end{equation}
Let $B''=\{w^\sharp>0\}$. We aim to control $|G\triangle B''|$ in terms of $\varepsilon$. 
By a simple decomposition,
\begin{equation}
\label{triangG}
|G\triangle B''|
= 2|G\setminus B''|
\leq 2|\{w>\overline t\}\setminus\{w^\sharp>\overline t\}|
+ 2|\{w\leq \overline t\}|.
\end{equation}
The first term is controlled by \eqref{epsgamm}, so it remains to estimate the second one. Using \eqref{distrw}, \eqref{buonocoreisoperimetrica1}, the bound $\zeta^{-1}\leq 1$, Corollary \ref{puntualedim2}, \eqref{muprimo=4pi}, and \eqref{claimbho}, we obtain
\begin{equation}
\label{listone}
\begin{aligned}
|\{ w \leq \overline t\}|
&= |G| - \mu_{\tilde u}(\overline t) \\
&\leq |G| - |\{w \geq \tilde u^*(|S|)\}|
- \int_{\overline t}^{\tilde u^*(|S|)} (-\mu'_{\tilde u}(s))\,ds \\
&\leq |G| - |S|
- \int_{\overline t}^{\tilde u^*(|S|)}
(-\mu'_{\tilde u}(s))\,\zeta^{-1}(\mu_{\tilde u}(s))\,ds \\
&\leq |G| - |S|
- \int_{\overline t}^{\tilde u^*(|S|)}
(-\mu'_{\tilde V}(s))\,(\nu^*)^{-1}(\mu_{\tilde V}(s))\,ds \\
&= |G| - |S| - 4\pi \tilde u^*(|S|) + 4\pi \overline t \\
&= 4\pi\bigl(\tilde V^*(|S|)-\tilde u^*(|S|)\bigr) + 4\pi \overline t \\
&\leq 2\sqrt{2\pi}\,\varepsilon^{\frac12}
+ 8\pi C\,\varepsilon^{\frac{\tilde\theta}{2}}.
\end{aligned}
\end{equation}
Combining \eqref{triangG} and \eqref{listone}, and observing that $\frac{\tilde\theta}{2}\leq\frac12$, we deduce
\begin{equation*}
|G\triangle B''| \leq C_1 \varepsilon^{\frac{\tilde\theta}{2}}.
\end{equation*}
We now argue similarly to control $|S\triangle B'|$, where $B'$ is the ball with $|B'|=|S|$ concentric with $\{w^\sharp>0\}$. From \eqref{measureofT}, we can find $\overline{\overline t}$ such that
\[
0 \leq \tilde u^*(|S|)-\overline{\overline t}
\leq 2C\varepsilon^{\frac{\tilde\theta}{2}}
\]
and
\begin{equation*}
|\{w>\overline{\overline t}\}\triangle\{w^\sharp>\overline{\overline t}\}|
< \varepsilon^{\frac{\tilde\theta}{2}}.
\end{equation*}
Then
\begin{equation}
\label{assS}
\begin{aligned}
|S\triangle B'|
&\leq |S\triangle\{w>\overline{\overline t}\}|
+ |\{w>\overline{\overline t}\}\triangle\{w^\sharp>\overline{\overline t}\}| \\
&\quad + |\{w^\sharp>\overline{\overline t}\}\triangle B'|.
\end{aligned}
\end{equation}
By the property of Lebesgue measure,
\[
|\{w^\sharp>\overline{\overline t}\}\triangle B'|=\abs{\abs{\{w>\overline{\overline{t}}\}^\sharp}-\abs{B'}}=\abs{\abs{\{w>\overline{\overline{t}}\}}-\abs{S}}
\leq |S\triangle\{w>\overline{\overline t}\}|,
\]
so it suffices to estimate $|S\triangle\{w>\overline{\overline t}\}|$. Since $\tilde S \subset \{w>\overline{\overline t}\}$, we have
\begin{equation*}
\begin{aligned}
|S\triangle\{w>\overline{\overline t}\}|
&\leq |S\setminus \tilde S|
+ |\{w>\overline{\overline t}\}| - |\{w>\overline{\overline t}\}\cap S| \\
&\leq |S\setminus \tilde S|
+ |\{w^\sharp>\overline{\overline t}\}| - |S\cap \tilde S|.
\end{aligned}
\end{equation*}
Using $S=(S\cap \tilde S)\cup(S\setminus \tilde S)$, we obtain
\begin{equation*}
\begin{aligned}
|S\triangle\{w>\overline{\overline t}\}|
&\leq |S\setminus \tilde S|
+ |\{w^\sharp>\overline{\overline t}\}| - |S|
+ |S\setminus \tilde S| \\
&\leq |\{\tilde V^*>\overline{\overline t}\}| - |S|
+ 2|S\setminus \tilde S|.
\end{aligned}
\end{equation*}
Here we used the pointwise bound $w^*(s)\leq \tilde V^*(s)$ for all $s\in(0,|G|)$. Since $|\{\tilde V^*>\overline{\overline t}\}|=|G|-4\pi \overline{\overline t}$, we get
\begin{equation*}
\begin{aligned}
|S\triangle\{w>\overline{\overline t}\}|
&\leq |G|-|S|-4\pi \overline{\overline t}
+ 2|S\setminus \tilde S| \\
&= 4\pi\bigl(\tilde V^*(|S|)-\overline{\overline t}\bigr)
+ 2|S\setminus \tilde S| \\
&= 4\pi\bigl(\tilde V^*(|S|)-\tilde u^*(|S|)\bigr)
+ 4\pi\bigl(\tilde u^*(|S|)-\overline{\overline t}\bigr)
+ 2|S\setminus \tilde S| \\
&\leq 2\sqrt{2\pi}\,\varepsilon^{\frac12}
+ 2C\,\varepsilon^{\frac{\tilde\theta}{2}}
+ 2\sqrt{\pi}\,\varepsilon^{\frac12}.
\end{aligned}
\end{equation*}
Combining this with \eqref{assS}, we conclude
\begin{equation*}
|S\triangle B'| \leq C_2 \varepsilon^{\frac{\tilde\theta}{2}}.
\end{equation*}
Finally, since
\[
(G\setminus S)\triangle(B''\setminus B')
\subset (G\triangle B'') \cup (S\triangle B'),
\]
we obtain
\[
\beta(\Omega)
\leq |(G\setminus S)\triangle(B''\setminus B')|
\leq |G\triangle B''| + |S\triangle B'|
\leq C\,\varepsilon^{\frac{\tilde\theta}{2}}.
\]
The conclusion follows by setting $\bar\theta=\frac{\tilde\theta}{2}$. 
Observe that $\bar\theta$ depends on the parameters $r,s$ appearing in Theorem \ref{polya_quant}.
\end{proof}

\section{Open Problems}
\label{section7}
%Unlike the case of the sharp quantitative Saint-Venant inequality, our result is only partial in the sense that the estimate does not account for a true index of asymmetry for the annulus.

Despite the result obtained in this work there are still some open question such as:
\begin{itemize}
\item is it possible to find the optimal exponent in Theorems \ref{thm:dueasimmetrie} and \ref{thm:anellia}?
\item Diaz and Weinstein established upper and lower bounds for the torsional rigidity of a beam, including in the case of multiply connected domains, in terms of its second-order momentum. It would be interesting to investigate a quantified version of this result.
\end{itemize}

\subsubsection*{Acknowledgements}
We would like to thank the anonymous referee for the valuable suggestions that helped improve the final version of this paper.

We would like to thank Prof. Cristina Trombetti for the valuable advice that helped us to achieve these results.
\subsubsection*{Declarations}

\paragraph{Funding}
The authors are members of GNAMPA and they were partially supported by Gruppo Nazionale per l’Analisi Matematica, la Probabilità e le loro Applicazioni
(GNAMPA) of Istituto Nazionale di Alta Matematica (INdAM).   

\paragraph{Data Availability} All data generated or analysed during this study are included in this published article.
	
\paragraph{Competing Interests} We declare that we have no financial and personal relationships with other people or organizations.

\bibliographystyle{plain}
\bibliography{bibliografia.bib}

\begin{thebibliography}{10}

\bibitem{ALT}
A.~Alvino, P.-L. Lions, and G.~Trombetti.
\newblock On optimization problems with prescribed rearrangements.
\newblock {\em Nonlinear Anal.}, 13(2):185--220, 1989.

\bibitem{AT}
A.~Alvino and G.~Trombetti.
\newblock Sulle migliori costanti di maggiorazione per una classe di equazioni ellittiche degeneri.
\newblock {\em Ricerche di Matematica}, 27(2):413--428, 1978.

\bibitem{AB}
V.~Amato and L.~Barbato.
\newblock Quantitative comparison results for first-order {Hamilton-Jacobi} equations.
\newblock {\em Acta Appl. Math.}, 200(1), December 2025.

\bibitem{ABCMP}
V.~Amato, R.~Barbato, S.~Cito, A.L. Masiello, and G.~Paoli.
\newblock A quantitative talenti-type comparison result with robin boundary conditions, 2025.

\bibitem{ABMP}
V.~Amato, R.~Barbato, A.L. Masiello, and G.~Paoli.
\newblock The {T}alenti comparison result in a quantitative form.
\newblock {\em Annali Scuola Normale Superiore - Classe di Scienze}, page~31, dec 2024.

\bibitem{AFP}
L.~Ambrosio, N.~Fusco, and D.~Pallara.
\newblock {\em Functions of Bounded Variation and Free Discontinuity Problems}.
\newblock Oxford University PressOxford, March 2000.

\bibitem{BS}
L.~{Barbato} and F.~{Salerno}.
\newblock {Talenti comparison results for solutions to $p$-Laplace equation on multiply connected domains}.
\newblock {\em https://doi.org/10.48550/arXiv.2504.06103}, 2025.

\bibitem{BM}
M.~F. Betta and A.~Mercaldo.
\newblock Uniqueness results for optimization problems with prescribed rearrangement.
\newblock {\em Potential Analysis}, 5:183--205, 1996.

\bibitem{BD}
L.~Brasco and G.~De~Philippis.
\newblock Spectral inequalities in quantitative form.
\newblock In {\em Shape {O}ptimization {A}nd {S}pectral {T}heory}, pages 201--281. De Gruyter Open, Warsaw, 2017.

\bibitem{BDV}
L.~Brasco, G.~De~Philippis, and B.~Velichkov.
\newblock Faber-{K}rahn inequalities in sharp quantitative form.
\newblock {\em Duke Math. J.}, 164(9):1777--1831, 2015.

\bibitem{BZ}
J.~E. Brothers and W.~P. Ziemer.
\newblock Minimal rearrangements of sobolev functions.
\newblock {\em Journal für die reine und angewandte Mathematik}, 384:153--179, 1988.

\bibitem{B}
P.~Buonocore.
\newblock Isoperimetric inequalities in the torsion problem for multiply connected domains.
\newblock {\em Z. Angew. Math. Phys.}, 36(1):47--60, 1985.

\bibitem{CH}
M.~Choulli and A.~Henrot.
\newblock Use of the domain derivative to prove symmetry results in partial differential equations.
\newblock {\em Math. Nachr.}, 192:91--103, 1998.

\bibitem{CFET}
A.~Cianchi, L.~Esposito, N.~Fusco, and C.~Trombetti.
\newblock A quantitative {P}\'{o}lya-{S}zeg\"{o} principle.
\newblock {\em J. Reine Angew. Math.}, 614:153--189, 2008.

\bibitem{CF2}
A.~Cianchi and N.~Fusco.
\newblock Functions of bounded variation and rearrangements.
\newblock {\em Archive for rational mechanics and analysis}, 165:1--40, 2002.

\bibitem{CL}
M.~Cicalese and G.~P. Leonardi.
\newblock A selection principle for the sharp quantitative isoperimetric inequality.
\newblock {\em Arch. Ration. Mech. Anal.}, 206(2):617--643, 2012.

\bibitem{CLI}
G.~Ciraolo and X.~Li.
\newblock A quantitative symmetry result for {$p$}-{L}aplace equations with discontinuous nonlinearities.
\newblock {\em Math. Ann.}, 392(2):2131--2155, 2025.

\bibitem{Simon}
S.~Cito.
\newblock Optimality and stability of the radial shapes for the {S}obolev trace constant.
\newblock {\em Nonlinear Analysis}, 269:114086, 2026.

\bibitem{CPP}
S.~Cito, G.~Paoli, and G.~Piscitelli.
\newblock A stability result for the first {R}obin-{N}eumann eigenvalue: a double perturbation approach.
\newblock {\em Commun. Contemp. Math.}, 27(6):Paper No. 2450039, 35, 2025.

\bibitem{CLM}
C.~Conca, A.~Laurain, and R.~Mahadevan.
\newblock Minimization of the ground state for two phase conductors in low contrast regime.
\newblock {\em SIAM J. Appl. Math.}, 72(4):1238--1259, 2012.

\bibitem{CMS}
C.~Conca, R.~Mahadevan, and L.~Sanz.
\newblock An extremal eigenvalue problem for a two-phase conductor in a ball.
\newblock {\em Applied Mathematics and Optimization}, 60(2):173--184, 2009.

\bibitem{DW}
J.~B. Diaz and A.~Weinstein.
\newblock The torsional rigidity and variational methods.
\newblock {\em Amer. J. Math.}, 70:107--116, 1948.

\bibitem{FMP2}
A.~Figalli, F.~Maggi, and A.~Pratelli.
\newblock A mass transportation approach to quantitative isoperimetric inequalities.
\newblock {\em Invent. Math.}, 182(1):167--211, 2010.

\bibitem{Fleming_Rishel}
W.H. Fleming and R.~Rishel.
\newblock An integral formula for total gradient variation.
\newblock {\em Arch. Math. (Basel)}, 11:218--222, 1960.

\bibitem{Fuglede}
B.~Fuglede.
\newblock Stability in the isoperimetric problem.
\newblock {\em Bull. London Math. Soc.}, 18(6):599--605, 1986.

\bibitem{F}
N.~Fusco.
\newblock The quantitative isoperimetric inequality and related topics.
\newblock {\em Bull. Math. Sci.}, 5(3):517--607, 2015.

\bibitem{FMP}
N.~Fusco, F.~Maggi, and A.~Pratelli.
\newblock The sharp quantitative isoperimetric inequality.
\newblock {\em Ann. of Math. (2)}, 168(3):941--980, 2008.

\bibitem{GPPS}
N.~Gavitone, G.~Paoli, G.~Piscitelli, and R.~Sannipoli.
\newblock An isoperimetric inequality for the first {S}teklov-{D}irichlet {L}aplacian eigenvalue of convex sets with a spherical hole.
\newblock {\em Pacific J. Math.}, 320(2):241--259, 2022.

\bibitem{H}
R.~R. Hall.
\newblock A quantitative isoperimetric inequality in {$n$}-dimensional space.
\newblock {\em J. Reine Angew. Math.}, 428:161--176, 1992.

\bibitem{HN}
W.~Hansen and N.~Nadirashvili.
\newblock Isoperimetric inequalities in potential theory.
\newblock In {\em Proceedings from the {I}nternational {C}onference on {P}otential {T}heory ({A}mersfoort, 1991)}, volume~3, pages 1--14, 1994.

\bibitem{HLP}
G.~H. Hardy, J.~E. Littlewood, and G.~P\'olya.
\newblock {\em Inequalities}.
\newblock Cambridge Mathematical Library. Cambridge University Press, Cambridge, 1988.
\newblock Reprint of the 1952 edition.

\bibitem{K}
S.~Kesavan.
\newblock {\em Symmetrization and applications}, volume~3.
\newblock world scientific, 2006.

\bibitem{Kim}
D.~Kim.
\newblock Quantitative inequalities for the expected lifetime of {B}rownian motion.
\newblock {\em Michigan Math. J.}, 70(3):615--634, 2021.

\bibitem{maggi2012sets}
F.~Maggi.
\newblock {\em Sets of finite perimeter and geometric variational problems}, volume 135 of {\em Cambridge Studies in Advanced Mathematics}.
\newblock Cambridge University Press, Cambridge, 2012.
\newblock An introduction to geometric measure theory.

\bibitem{MS}
A.L. Masiello and F.~Salerno.
\newblock A quantitative result for the {$k$}-{H}essian equation.
\newblock {\em Nonlinear Anal.}, 255:Paper No. 113776, 2025.

\bibitem{MNP}
I.~Mazari, G.~Nadin, and Y.~Privat.
\newblock Shape optimization of a weighted two-phase {D}irichlet eigenvalue.
\newblock {\em Arch. Ration. Mech. Anal.}, 243(1):95--137, 2022.

\bibitem{PPS}
G.~Paoli, G.~Piscitelli, and R.~Sannipoli.
\newblock A stability result for the {S}teklov {L}aplacian eigenvalue problem with a spherical obstacle.
\newblock {\em Commun. Pure Appl. Anal.}, 20(1):145--158, 2021.

\bibitem{PPT}
G.~Paoli, G.~Piscitelli, and L.~Trani.
\newblock Sharp estimates for the first {$p$}-{L}aplacian eigenvalue and for the {$p$}-torsional rigidity on convex sets with holes.
\newblock {\em ESAIM Control Optim. Calc. Var.}, 26:Paper No. 111, 15, 2020.

\bibitem{P}
G.~P\'olya.
\newblock Torsional rigidity, principal frequency, electrostatic capacity and symmetrization.
\newblock {\em Quart. Appl. Math.}, 6:267--277, 1948.

\bibitem{PS}
G.~P\'olya and G.~Szeg\"o.
\newblock {\em Isoperimetric {I}nequalities in {M}athematical {P}hysics}, volume No. 27 of {\em Annals of Mathematics Studies}.
\newblock Princeton University Press, Princeton, NJ, 1951.

\bibitem{PW}
G.~Polya and A.~Weinstein.
\newblock On the torsional rigidity of multiply connected cross-sections.
\newblock {\em Ann. of Math. (2)}, 52:154--163, 1950.

\end{thebibliography}
\Addresses
\end{document}